\documentclass[12pt,letter]{article}  
\usepackage{amssymb}
\usepackage{amsthm}
\usepackage{amsfonts}
\usepackage{amsmath}
\usepackage{anysize}
\usepackage{bbm} 
\marginsize{2.5cm}{2.5cm}{1cm}{1.6cm}
\usepackage{color}

\usepackage{hyperref}
\hypersetup{        
    unicode=false,         
    pdftoolbar=true,       
    pdfmenubar=true,       
    pdffitwindow=false,    
    pdfstartview={FitH},   
    pdfauthor={T. Darvas, C.H. Lu},     
    colorlinks=true,       
   linkcolor=black,         
    citecolor=black,        
    filecolor=black,      
    urlcolor=black}          

\newtheorem*{theorem*}{Theorem}
\newtheorem*{claim*}{Claim}

\newtheorem{theorem}{Theorem}[section]

\newtheorem{lemma}[theorem]{Lemma}
\newtheorem{conj}[theorem]{Conjecture}
\newtheorem{corollary}[theorem]{Corollary}
\newtheorem{remark}[theorem]{Remark}

\newtheorem{proposition}[theorem]{Proposition}

\def\i{\sqrt{-1}}
\def\del{\partial}
\def\dbar{\bar\partial}
\def\ddbar{\del\dbar}

\def\del{\partial}

\DeclareMathOperator{\Ent}{Ent}
\DeclareMathOperator{\PSH}{PSH}

\def\o{\omega}

\newcommand{\Ecp}{\mathcal{E}^p_{\omega}}
\newcommand{\setdef}{\; ; \; }

\title{Geodesic stability, the space of rays, and uniform convexity in Mabuchi geometry}
\author{Tam\'as Darvas and Chinh H. Lu\vspace{-0.1in}}
\date{}

\begin{document}

\maketitle

\begin{abstract} We establish the essentially optimal form of Donaldson's geodesic stability conjecture regarding existence of constant scalar curvature K\"ahler metrics. We carry this out by exploring in detail the metric geometry of  Mabuchi geodesic rays, and the uniform convexity properties of the space of K\"ahler metrics. 
\end{abstract}

\section{Introduction}

In this paper we prove the essentially optimal form of Donaldson's geodesic stability conjecture \cite{Do99}. The main result is obtained via a detailed analysis  of the rays associated to the space of K\"ahler metrics. 

Suppose $(X,\omega)$ is a compact K\"ahler manifold with $\dim X =n$. We  consider $\mathcal H$, the space of K\"ahler metrics cohomologous to $\omega$, with its $L^p$ type Mabuchi metric structures $(\mathcal H,d_p), \ p \geq 1$ \cite{Da15}. For simplicity, to describe our motivation, let us momentarily assume that $X$ has no non-trivial holomorphic vector fields. In the recent breakthrough papers \cite{ChCh1,ChCh2,ChCh3} Chen--Cheng provided the first existence theorems of constant scalar curvature K\"ahler (csck) metrics inside the class $\mathcal H$. Such metrics are minimizers of Mabuchi's K-energy functional $\mathcal K:\mathcal H \to \mathbb{R}$ \cite{Ma87}. Together with \cite{BDL16}, the Chen--Cheng results provided a full characterization of existence of csck metrics  in terms of $d_1$-properness of $\mathcal K$. As $d_1$-properness is actually equivalent with properness in terms of Aubin's $J$-functional \cite{Da15}, this also verified an old conjecture of Tian  \cite{Ti94},\cite[Conjecture 7.12]{Ti00}, with the precise statement appearing in \cite[Conjecture 2.8]{DR17}.  

Energy properness is the strongest form of stability. Contrasting this is uniform K-stability, one of the weakest such conditions. When the K\"ahler structure is induced by an ample line bundle, this criterion was first considered by Sz\'ekelyhidi \cite{Sz06}, and was further studied by Dervan, Berman--Boucksom--Jonsson, Boucksom--Hisamoto--Jonsson \cite{De16,BBJ15,BHJ16,BHJ17} and many others. The ultimate hope is that (uniform) K-stability is weak enough to be verified using computational techniques of algebraic geometry, this being the main motivation behind the  Yau--Tian--Donaldson (YTD) conjecture, seeking to show that some form of K-stability is equivalent with existence of csck metrics. 

In this paper we focus on Donaldson's geodesic stability conjecture \cite[Conjecture 12]{Do99},  lying between energy properness and uniform K-stability (see Conjecture \ref{conj: geod_stability} below). This conjecture  predicts that it is enough to check properness of the K-energy along the geodesic rays of $\mathcal H$ to insure existence of csck metrics. Initially, the predictions of Donaldson  advocated for the use of  smooth geodesic rays \cite{Do99}. As we know now, the typical regularity of geodesics is merely $C^{1,1}$ \cite{Ch00, Bl13, DL12, CTW17}, even when connecting smooth endpoints. Hence the present expectation is that (in its optimal form) Donaldson's geodesic stability conjecture should hold for rays that have at most two bounded derivatives. In Theorem \ref{thm: Linftygeod_stability_intr} we essentially verify this form of the conjecture.

To carry out our plan, we  first explore in depth the metric geometry of $L^p$ geodesic rays (i.e. rays running inside the $d_p$-completions of $\mathcal H$), a topic of independent interest. To do this, perhaps surprisingly, we need to understand uniform convexity of the $L^p$ Mabuchi geometry when $p>1$, extending work of Calabi--Chen in the particular case $p=2$ \cite{CC02}.  After exploring the metric space of $L^p$ geodesic rays, we show that such rays can always be approximated via rays  of $C^{1,\bar 1}$ potentials, \emph{with} converging radial K-energy. With slightly different formulation, the uniform $L^1$ geodesic stability conjecture was verified in \cite{ChCh2,ChCh3}, pointing out that it is enough to test energy properness along $L^1$ geodesic rays  to guarantee existence of csck metrics. This result, together with our approximation theorems just mentioned will yield the geodesic stability theorem for rays of $C^{1, \bar 1}$ potentials, i.e., potentials with bounded complex Hessian.

In addition to the above, our results resolve a number of related open questions in K\"ahler geometry, specified in the paragraphs below. Also, in the particular case when the K\"ahler structure is induced by an ample line bundle, our theorems also make connection with the variational program designed to attack the uniform YTD conjecture (see \cite{Bo18,ChCh3}). Roughly speaking, to verify the uniform YTD conjecture using our results, it is now enough to show that specific $C^{1, \bar 1}$ geodesic rays can be approximated by geodesic rays induced by the so called test configurations of algebraic geometry \cite{Ti97,Do02} (with converging radial K-energy). On the surface this sounds simpler than approximating $L^1$ metric geodesic rays \cite[p.2]{Bo18}, and time will tell what role our results will play in the solution of this problem.
\vspace{-0.3cm}

\paragraph{Uniform convexity and uniqueness of geodesic segments. } By $\mathcal H_{\omega}$ we denote the space of K\"ahler potentials associated to $\mathcal H$.  The metric completions of $(\mathcal H_{\omega},d_p)$ are $(\mathcal E^p_{\omega},d_p)$, and the latter spaces are complete geodesic metric spaces for any $p \geq 1$ \cite{Da15}. The distinguished $d_p$-geodesics running between the points of $\mathcal E^p_{\omega}$ are called $L^p$ \emph{finite energy geodesics} (or simply finite energy geodesics, or $L^p$ geodesics, if no confusion arises). These curves arise as limits of solutions to  degenerate equations of complex Monge--Amp\`ere type. We recall the basic properties of these spaces in Section  \ref{subsect: Lp finsler}. 

For any $p \in [1,\infty)$ it was shown in \cite[Theorem 1.5]{ChCh3} that the metrics $d_p$ are ``convex": if $[0,1] \ni t \to u_t,v_t \in \mathcal E^p$ are two finite energy geodesic segments then  
\begin{equation}\label{eq: ChCh3_metric_convex}
d_p(u_\lambda,v_\lambda) \leq (1-\lambda)d_p(u_0,v_0) + \lambda d_p(u_1,v_1), \ \lambda  \in [0,1].
\end{equation}
This property is called Buseman convexity in the metric geometry literature \cite[Section 2.2]{Jo97}, going back to \cite{Bu55}. In the particular case $p=1$, \eqref{eq: ChCh3_metric_convex} was established in \cite[Proposition 5.1]{BDL17}, having applications to the convergence of the weak Calabi flow. In case $p=2$, \eqref{eq: ChCh3_metric_convex} follows from the fact that $(\mathcal E^2_{\omega},d_2)$ is a complete CAT(0) metric space, as shown in \cite[Theorem 1]{Da17}, building on estimates of \cite[Theorem 1.1]{CC02}. 

The CAT(0) property consists of the following estimate: if $u \in \mathcal E^2_{\omega}$ and $[0,1] \ni t \to v_t \in \mathcal E^2_{\omega}$ is a finite energy geodesic segment then 
\begin{equation}\label{eq: CAT(0)}
d_2(u,v_\lambda )^2 \leq (1-\lambda)d_2(u,v_{0})^2 +\lambda d_2(u,v_1)^2 - \lambda(1-\lambda )d_2(v_0,v_1)^2, \ \lambda \in [0,1].
\end{equation}
As is well known, \eqref{eq: CAT(0)} implies \eqref{eq: ChCh3_metric_convex} \cite[Prop 2.3.2]{Jo97}. Unfortunately, there is very strong evidence that \eqref{eq: CAT(0)} can not hold for the  $d_p$ metrics when $p \neq 2$. Indeed, when restricting to a toric K\"ahler manifold and toric K\"ahler metrics, the spaces $(\mathcal E^p_\o,d_p)$ are isometric to the flat $L^p$ metric spaces of convex functions  defined on a  convex polytope of $\mathbb{R}^n$ \cite[Section 6]{DG16}. It is well known however that  CAT(0) Banach spaces are in fact Hilbert spaces \cite{BH99}, evidencing that only $(\mathcal E^2,d_2)$ can be CAT(0). 

Despite this, in the first main result of this paper we show that adequate generalizations of the CAT(0) inequality \eqref{eq: CAT(0)} do hold for the $d_p$ metrics, in case $p>1$. These can be viewed as the K\"ahler analogs of classical inequalities of Clarkson and Ball--Carlen--Lieb, regarding the uniform convexity of $L^p$ spaces \cite{Cl36,BCL94}. Consequently, the metric spaces $(\mathcal E^p_{\omega},d_p)$ are \emph{uniformly convex} for $p>1$, giving them extra structure that will be explored in the latter parts of the paper:

\begin{theorem}  \label{thm: uniform_convex_intr} Let $p \in (1,\infty)$. Suppose that $u \in \mathcal E^p_{\omega}, \ \lambda \in [0,1]$ and $[0,1] \ni t \to v_t \in \mathcal E^p_{\omega}$ is a finite energy geodesic segment. Then the following hold:\vspace{0.1cm}\\
(i) $d_p(u,v_\lambda )^2 \leq (1-\lambda)d_p(u,v_{0})^2 +\lambda d_p(u,v_1)^2 - (p-1)\lambda(1-\lambda )d_p(v_0,v_1)^2$, if $1 < p \leq 2.$\\
(ii) $d_p(u,v_\lambda)^p \leq (1-\lambda )d_p(u,v_{0})^p +\lambda d_p(u,v_1)^p - \lambda^{\frac{p}{2}} (1-\lambda)^{\frac{p}{2}}d_p(v_0,v_1)^p$, if $2 \leq  p.$
\end{theorem}
In the particular case $p=2$ this result recovers the inequalities of Calabi--Chen \cite{CC02}, however our proof of Theorem \ref{thm: uniform_convex_intr} is very different from the argument in \cite{CC02}, as the differentiation of $d_p$ metrics is problematic for $p \neq 2$. 

It was pointed out in the comments following \cite[Theorem 4.17]{Da15} that  $d_1$-geodesic segments connecting the different points of $(\mathcal E^1_{\omega},d_1)$ are not unique. However, as a consequence of the above result it follows that uniqueness of $d_p$-geodesic segments does hold in case $p>1$: 
\begin{theorem}\label{thm: unique_geod_intr} Let $p \in (1,\infty)$, and suppose that $[0,1] \ni t \to v_t \in \mathcal E^p_{\omega}$ is the $L^p$ finite energy geodesic connecting $v_0,v_1 \in \mathcal E^p_\o$. Then $t \to v_t$ is the only $d_p$-geodesic connecting $v_0,v_1$, i.e., $(\mathcal E^p_{\omega},d_p)$ is a uniquely geodesic metric space.\vspace{-0.3cm}
\end{theorem}

\paragraph{The metric geometry of geodesic rays.} Next we explore the metric geometry of $\mathcal R^p_u$, the space of finite energy $L^p$ geodesic rays emanating from a fixed potential $u \in \mathcal E^p_{\omega}$. As a convention, given $p \in [1,\infty),$ a finite energy geodesic ray $[0,\infty) \ni t \to u_t \in \mathcal E^p_{\omega}$ with $u_0 =u$ will be simply denoted by $\{ u_t\}_t \in \mathcal R^p_u$.

In accordance with the metric space literature, two $d_p$-rays $[0,\infty) \ni t \to u_t,v_t \in \mathcal E^p_{\omega}$ are parallel/asymptotic if $d_p(u_t,v_t)$ is uniformly bounded for $t \geq 0$ \cite[Chapter II.8]{BH99}. To start, we point out in Proposition \ref{prop: parallel} that for any $v \in \mathcal E^p_{\omega}$ and $\{u_t\}_t \in  \mathcal{R}^p_u$ it is possible to find a unique $\{ v_t\}_t \in \mathcal R^p_v$ such that $\{u_t\}_t$ and $\{v_t\}_t$ are parallel. Consequently, the $d_p$-geometries  verify Euclid's 5th postulate for half-lines, answering an open question of Chen--Cheng \cite[Remark 1.6]{ChCh3}, who proved this for $p=1$ under restrictive conditions on the slope of the K-energy along $\{u_t\}_t$. Thus, we can introduce a natural parallelism operator $\mathcal P_{uv}:\mathcal R^p_u \to \mathcal R^p_v$ for any $u,v \in \mathcal E^p_{\omega}$. Moreover it is possible to introduce natural metric structures on $\mathcal R^p_u$ and $\mathcal R^p_v$ making this map an isometry:

\begin{theorem} Let $p \in [1,\infty)$. For any $u \in \mathcal E^p_{\omega}$, $(\mathcal R^p_u,d^c_{u,p})$ is a complete metric space. For any $v \in \mathcal E^p_{\omega}$ the parallelism operator $\mathcal P_{uv}: (\mathcal R^p_u,d^c_{u,p}) \to (\mathcal R^p_v,d^c_{v,p})$ is an isometry. 
\end{theorem}

In this result, the $d^c_{u,p}$ metric is called the \emph{chordal $L^p$ metric} between two rays, defined by the following expression:
\begin{equation}\label{eq: dpc_def}
d_{u,p}^c(\{u_t\}_t,\{v_t\}_t):= \lim_{t \to \infty} \frac{d_p(u_t,v_t)}{t}, \ \ \{ u_t\}_t \in \mathcal{R}^p_{u},\ \{ v_t\}_t \in \mathcal R^p_{u}.
\end{equation}
That this limit exists and is finite follows from \eqref{eq: ChCh3_metric_convex}. Though not necessarily treated as a metric in other works, \cite[Corollary 5.6]{ChCh3}, \cite[Formula 1.2]{Bo18} also consider the expression on the right hand side of \eqref{eq: dpc_def}, in the slightly restrictive case of unit speed geodesic rays, and non-Archimedean metrics respectively (see also  \cite[Lemma 3.1]{BDL16}). Moreover, one would think that the metrics of the graded filtrations defined in \cite[Section 3]{BJ18} should be related to the above concept as well.

It was pointed out recently that $L^1$ Mabuchi geometry can be defined for big classes as well \cite{DDL3}. Using this, it is possible to introduce the metric space of weak $L^1$ rays in the big context (see \cite{DDL5} where we embed singularity types into the space of $L^1$ rays).

By the last part of the above theorem, there is no new information gained by considering different starting points for rays, hence it makes sense to restrict attention to the space $(\mathcal R^p_{\omega},d^c_{p})$, representing the space of rays emanating from $0 \in \mathcal H_{\omega}$.
The above theorem points out that $d_{p}^c$ thus defined gives a complete metric on the space of \emph{all} $L^p$ rays emanating from a fixed starting point, that includes the constant ray. In our next main result we point out that the resulting metric spaces have rich geometry:

\begin{theorem}\label{thm: R_p_complete_geod_K_convex_intr} $(\mathcal R^p_{\omega},d^c_{p})$ is a geodesic metric space for any $p \in [1,\infty)$. Additionally, the radial K-energy is convex along $d^c_p$-geodesic segments.
\end{theorem}

The \emph{radial K-energy} is defined for any $\{u_t\}_t \in \mathcal R^p_{\omega}$,  and is given by the expression
$$\mathcal K\{u_t\} := \lim_{t\to \infty}\frac{\mathcal K(u_t)}{t},$$ 
where $\mathcal K: \mathcal E^p_{\omega} \to (-\infty,\infty]$ is the extended K-energy of Mabuchi from \cite{BBEGZ11,BDL17}. The radial K-energy is $d_p^c$-lsc, possibly equal to $\infty$, and in the setting of unit speed geodesics, its definition agrees with the $\yen$ invariant of \cite{ChCh3}. Also, there is clear parallel with the non-Archimedean K-energy (see \cite{Bo18} and references therein). 

This theorem represents the radial version of \cite[Theorem 2]{Da15} and \cite[Theorem 1.2]{BDL17} (building on \cite{BrBn17}). 
In slight contrast with previous speculations in the literature (see for example \cite{BHJ17} or \cite[Definition 1.8]{ChCh3})  it seems more natural to consider the space of \emph{all} $d_p$-rays, not just the  ones that have $d_p$-unit speed. Allowing for a bigger class of rays makes possible the construction of $d^c_p$-geodesic segments running between any two points of $\mathcal R^p_\o$, with good convexity properties. Moreover, the convexity of the radial K-energy on $\mathcal R^p_\o$ could potentially be used to set up the study of optimal degenerations as a convex optimization problem (see \cite{DS16}).

The $d^c_p$-geodesic segments constructed in the proof of the above theorem are called $d^c_p$-\emph{chords}, as they are reminiscent of the classical chords in the chordal geometry of the unit sphere of $\mathbb{R}^n$ (at least when restricting to $d_p$-unit speed rays). In case $p>1$, due to uniform convexity (Theorem \ref{thm: uniform_convex_intr}),  we will construct the $d^c_p$--chords  directly. In case $p=1$, in the absence of uniform convexity, the construction of $d^c_1$-chords is done using an approximation procedure, via our next main theorem.

We have $\mathcal R^p_\o \subset \mathcal R^{p'}_\o$ for any $p' \leq p$. More importantly, by the proof of Theorem \ref{thm: R_p_complete_geod_K_convex_intr}, $d^c_p$-chords are automatically $d^c_{p'}$-chords as well, giving further evidence that it is more advantageous to consider the space of all rays, not just the ones with $d_p$-unit speed. This latter fact again represents the radial version of a well known phenomenon for the family of metric spaces $(\mathcal E^p_\o,d_p), \ p \geq 1$, according to which geodesics are ``shared" when comparing different classes. Though the space of $d_p$-unit speed rays seems to exhibit a metric structure reminiscent of the Tits geometry attached to CAT(0) spaces \cite{BH99}, none of the above  properties hold for these structures.

Next we turn to approximation. The collection of geodesic rays $\{u_t\}_t \in \mathcal R^1_{\omega}$ with $u_t \in L^\infty, \ t \geq 0$ will be denoted by $\mathcal R^\infty_{\omega}$, and will be referred to as the set of \emph{geodesic rays with bounded potentials}. In addition to having bounded potentials, the rays of $\mathcal R^\infty_\o$ are actually $t$-Lipschitz, and they solve the geodesic equation of $L^p$ Mabuchi geometry in the weak Bedford--Taylor sense, as opposed to the rays of $\mathcal R^p_\o, \ p \in [1,\infty)$, that are only limits of solutions to such equations (See Section 2.1).  

By $\mathcal H^{1, \bar 1}_\o$ we will denote the set of potentials in $\textup{PSH}(X,\omega)$ whose Laplacian (or whose complex Hessian) is bounded. Analogously, the collection of geodesic rays $\{u_t\}_t \in \mathcal R^1_{\omega}$ with $u_t \in \mathcal H^{1, \bar 1}_\o, \ t \geq 0$ will be denoted by $\mathcal R^{1, \bar 1}_{\omega}$, and will be referred to as the set of \emph{geodesic rays with $C^{1, \bar 1}$ potentials}. The space of rays with bounded Hessian, denoted by $\mathcal R^{1,1}_\omega$, is defined similarly.

The next result points out that $\mathcal R^\infty_\omega$ is $d_p^c$-dense in $\mathcal R^p_{\omega}$ for any $p \in [1,\infty)$. Also, we show that $\mathcal R^{1, \bar 1}_\o$ dense among rays with finite radial K-energy. In both cases one can approximate with converging radial K-energy:

\begin{theorem}\label{thm: approximation_dp_intr}
Let $\{ u_t\}_t \in \mathcal R^p_{\omega}$ with $p \in [1,\infty)$. The following hold:\\
(i) There exists a sequence $\{ u^j_t\}_t \in \mathcal R^\infty_{\omega}$ such that $u^j_t \searrow u_t, \ t \geq 0$,  $d_p^c(\{u_t^j\}_t,\{u_t\}_t) \to 0$ and $\mathcal{K}\{u_t^j\} \to  \mathcal{K}\{u_t\}$. \\
(ii) If $\mathcal K\{u_t\}<\infty$, then there exists a sequence $\{ v^j_t\}_t \in \mathcal R^{1, \bar 1}_{\omega}$ such that $v^j_t \searrow  u_t, \ t \geq 0$,  $d_p^c(\{v_t^j\}_t,\{u_t\}_t) \to 0$ and $ \mathcal{K}\{v_t^j\} \to  \mathcal{K}\{u_t\}$.
\end{theorem}

It remains to be seen if the condition $\mathcal K\{u_t\} < \infty$ can be omitted in (ii). This theorem can be seen as a radial analog of \cite[Theorem 1.3]{BDL17}, perhaps also making progress on the variational program designed to attack the uniform YTD conjecture (see step (4) in \cite[p. 2]{Bo18}, c.f. \cite[Conjecture 2.5]{BJ18}). Time will tell exactly how our results will fit into this program, but now it is enough to show that some $C^{1,\bar 1}$ rays can be approximated by rays induced by test configurations  (with converging K-energy) to prove the uniform YTD conjecture.

As a first step in obtaining Theorem \ref{thm: approximation_dp_intr}(i), in Theorem \ref{thm: approximation of d1 rays} we show that one can approximate by bounded geodesic rays with possibly diverging radial K-energy. The argument uses \cite{RWN14}, and this will suffice in case $\mathcal K\{u_t\} = \infty,$ since  $\mathcal K\{\cdot\}$ is $d_p^c$-lsc. However to obtain (i) in case $\mathcal K\{u_t\}$ is finite, a much more delicate construction will be needed, building on the relative Ko{\l}odziej type estimate of \cite{DDL4}. The proof of (ii)  builds on (i), and novel apriori $C^{1, \bar 1}$ estimates along geodesic segments that are ``scalable" along rays. These will be obtained using the framework of \cite{He15} and \cite{GZ17}.

\vspace{-0.3cm}

\paragraph{Applications to geodesic stability.} We point out applications  to existence of constant scalar curvature K\"ahler (csck) metrics in terms of geodesic stability, going back to Donaldson's early conjectures in  \cite{Do99}.

To start, we say that $(X,\omega)$ is \emph{geodesically} $L^p/C^{1, \bar 1}$\emph{-semistable} if for any $\{u_t\}_t \in \mathcal R^p_{\omega}/ \mathcal R^{1,\bar 1}_\omega$ we have that $\mathcal K\{u_t\} \geq 0$ for $p \in [1,\infty]$. Regarding the relevance of semistability for the csck continuity method, we refer to \cite{ChCh3}. As an immediate consequence of Theorem \ref{thm: approximation_dp_intr} we obtain the following:

\begin{theorem}\label{thm: semistable_intr} $(X,\omega)$ is geodesically $L^1$-semistable if and only if it is geodesically $C^{1, \bar 1}$-semistable.\end{theorem}

Let $G:= \textup{Aut}_0(X)$ be the identity component of the group of holomorphic automorphisms of $X$. By $I: \mathcal E^1_\omega \to \mathbb{R}$ we denote the Monge--Amp\`ere energy functional (sometimes called Aubin--Yau or Aubin--Mabuchi energy). Then, as explained in \cite{DR17}, $G$ induces an isometry on $\mathcal E^1_0 = \mathcal E^1_{\omega} \cap I^{-1}(0)$, and one can introduce the following pseudo-metric on the orbits $\mathcal E^1_0 / G$:
$$d_{1,G}(G u_0,G u_1) := \inf_{g \in G} d_{1}(u_0,g.u_1).$$
Moreover, one can analogously define the space of \emph{normalized rays} $\mathcal R^p/\mathcal R^{1,\bar 1}/\mathcal R^{1, 1}, \ \ p \in [1,\infty]$,  where we restrict to rays  $\{u_t\}_t \in \mathcal R^p_{\omega}/\mathcal{R}^{1,\bar{1}}_{\omega}/\mathcal{R}^{1,1}_{\omega}$  with $I (u_t)=0, \ t \geq 0$.

By showing that minimizers of the K-energy on $\mathcal E^1_{\omega}$ are actually smooth csck potentials \cite[Theorem 1.5]{ChCh2}, Chen--Cheng have verified the last remaining condition of the existence/properness principle of \cite{DR17}, applied to the case of csck metrics. Together with the necessity result \cite[Theorem 1.5]{BDL16}, their theorem showed that existence of csck metrics is equivalent with properness of $\mathcal K$ in the following sense: there exists $\delta, \gamma > 0$ such that
\begin{equation}\label{eq: CC3BDL16_ineq}
\mathcal K(u) \geq \delta d_{1,G}(G0,Gu) - \gamma, \ \ \ u \in \mathcal E^1_\omega,
\end{equation}
Clearly, $d_{1,G}(G v_0,G v_1) \leq d_1(v_0,v_1), \ v_0,v_1 \in \mathcal E^1_\o$, and we say that $\{u_t\}_t \in \mathcal R^1$ is $G$\emph{-calibrated} if the curve $t \to Gu_t$  is a $d_{1,G}$-geodesic with the same speed as $\{u_t\}_t$, i.e.,
$$d_{1,G}(Gu_0,Gu_t)=d_1(u_0,u_t), \ \ \  t \geq 0.$$
Geometrically, $\{u_t\}_t$ is $G$-calibrated if it cuts each $G$-orbit inside $\mathcal E^1_{\omega}$ ``perpendicularly". In case $G = \{Id\}$, every ray is $G$-calibrated. 

Building on these concepts, it is natural to state the $C^{1, 1}$ uniform analog to Donaldson's geodesic stability conjecture, with the original formulation of \cite[Conjecture 12]{Do99} more closely related to the language of ``polystability":

\begin{conj}[$C^{1,1}$ uniform geodesic stability]\label{conj: geod_stability} Let $(X,\omega)$ be a compact K\"ahler manifold. Then the following are equivalent:\vspace{0.1cm}\\
(i) There exists a csck metric in $\mathcal H$.\vspace{0.1cm}\\
(ii) There exists $\delta >0$ such that $\mathcal K\{u_t\} \geq \delta \limsup_t\frac{d_{1,G}(G 0,G u_t)}{t}$ for all  geodesic rays $\{u_t\}_t \in \mathcal R^{1,  1}$.\vspace{0.1cm}\\
(iii) $\mathcal K$ is $G$-invariant and there exists $\delta >0$ such that for all $G$-calibrated geodesic rays $\{u_t\}_t \in \mathcal R^{1,  1}$ we have that $\mathcal K\{u_t\} \geq \delta d_1(0,u_1).$ 
\end{conj}

The statement of (ii) clearly points out that uniform geodesic stability is simply the condition that tests energy properness (expressed in \eqref{eq: CC3BDL16_ineq}) along a class of geodesic rays.

As the notion of $G$-calibrated rays has an obvious analog in case of the space of finite dimensional rays as well (within the context of K\"ahler quantizaton), we included this condition here to perhaps facilitate in the future an alternative definition for uniform K-stability in the presence of vector fields.

As explained in \cite[Proposition 5.5]{DR17}, in the above conjecture the $d_1$ distance is interchangeable with Aubin's $J$ functional. Lastly, given that rays induced by 1-parameter actions of $G$ are never $G$-calibrated, the condition that $\mathcal K$ is $G$-invariant (equivalent to vanishing Futaki invariant \cite{Fu83}) is necessary in the statement of (iii).

Using our above theorems, we prove in Theorem \ref{thm: L1Linft_geod_eqv} and Theorem \ref{thm: L1Linft_geod_calibrated_eqv} that the $C^{1, \bar 1}$ and $L^1$ version of the uniform geodesic stability conjecture are equivalent.
As alluded to previously, the breakthrough of Chen--Cheng \cite{ChCh2,ChCh3} together with  \cite[Theorem 4.7]{Da18}  essentially yielded the $L^1$ version of this conjecture (see Theorem \ref{thm: L1geod_stability} below). Putting things together, we arrive at our most important main result, essentially settling Conjecture \ref{conj: geod_stability}:
\begin{theorem}[$C^{1, \bar 1}$ uniform geodesic stability]\label{thm: Linftygeod_stability_intr} Let $(X,\omega)$ be a compact K\"ahler manifold. Then the following are equivalent:\vspace{0.1cm}\\
(i) There exists a csck metric in $\mathcal H$.\vspace{0.1cm} \\
(ii) There exists $\delta >0$ such that $\mathcal K\{u_t\} \geq \delta \limsup_t\frac{d_{1,G}(G 0,G u_t)}{t}$ for all  $\{u_t\}_t \in \mathcal R^{1, \bar 1}$.\vspace{0.1cm}\\
(iii) $\mathcal K$ is $G$-invariant and there exists $\delta >0$ s.t. $\mathcal K\{u_t\} \geq  \delta d_1(0,u_1) $  for all $G$-calibrated geodesic rays $\{u_t\}_t \in \mathcal R^{1, \bar 1}$.
\end{theorem}

Clearly, given the obvious inclusions among classes or geodesic rays, the $L^p$ versions of Conjecture \ref{conj: geod_stability} follow as well (with $\mathcal R^p$ replacing $\mathcal R^{1, \bar 1}$ in the statement). 
Though slightly different in formulation, the $L^\infty$ version of this result essentially confirms the equivalences between the conditions $(3),(4)$ and $(5)$ in  \cite[Question 1.12]{ChCh3} (see also the closely related questions of \cite[Remark 1.3]{ChCh2}). In case $G=\{Id\}$, the statement of the theorem can be made especially simple:
\begin{theorem}
\label{thm: Linftygeod_stability_Gtriv_intr} Let $(X,\omega)$ be a compact K\"ahler manifold without non-trivial holomorphic vector fields. Then the following are equivalent:\vspace{0.1cm}\\
(i) There exists a csck metric in $\mathcal H$.\vspace{0.1cm}\\
(ii) There exists $\delta>0$ such that $\mathcal K\{u_t\} \geq \delta d_1(0,u_1)$ for all $\{u_t\}_t \in \mathcal R^{1, \bar 1}$.
\end{theorem}

It remains to be seen if in the above stability results one can use rays that have potentials with fully bounded Hessian, not just bounded complex Hessian.  Even if possible, this small improvement seems to require substantial amount of new work.
Further optimizations are extremely unlikely, given that the typical regularity of geodesics breaks down beyond $C^2$ estimates. One would think that generalizations to the context of extremal and conical type csck metrics should be possible, using our results together with \cite{He18,Zh18}.\vspace{-0.3cm}

\paragraph{Connections with the literature.} 
Uniform convexity of metric spaces is an active area of research (see \cite{Oh07, Ke14,Ku14, NS11} and references therein). In particular, by \cite[Proposition 2.5]{Ku14} the inequalities of Theorem \ref{thm: uniform_convex_intr} are essentially optimal.

The notion of K-stability goes back to work of Tian \cite{Ti97}, with generalizations and precisions made along the way by S. Donaldson \cite{Do02}, Li--Xu \cite{LX14}, G. Sz\'ekelyhidi \cite{Sz06} and many others. Though the precise form of K-stability is still not fully clarified for general K\"ahler manifolds \cite{ACGTF}, at least in the absence of non-trivial holomorphic vector fields, it is widely expected that uniform K-stability will be equivalent with existence of csck metrics (see \cite[Question 1.12]{ChCh3}, \cite[Conjecture 4.9]{Bo18}). Informally, uniform K-stability simply says that Conjecture \ref{conj: geod_stability} holds for $C^{1,\bar 1}$ rays that are induced by the so called test configurations of $(X,\omega)$.

Closing the gap between $L^1$ uniform geodesic stability and uniform K-stability is the last remaining step in the variational program designed to attack the uniform YTD conjecture (see \cite[p.2]{Bo18}), with our Theorem \ref{thm: Linftygeod_stability_intr} representing an intermediate step. To facilitate further progress in this direction, based on the findings of Theorem \ref{thm: R_p_complete_geod_K_convex_intr}, one possible approach would be to develop the radial analog of the K\"ahler quantization scheme, recently extended to the $d_p$-metric completions in \cite{DLR18} (building on prior work by Berndtsson \cite{Bn09}, Chen--Sun \cite{CS12}, Donaldson \cite{Do02,Do05}, Phong--Sturm \cite{PS06}, Song--Zeldtich \cite{SZ10}, Tian \cite{Ti90} and others). Indeed, in case the K\"ahler structure $(X,\omega)$ is induced by an ample Hermitian line bundle $(L,h)$, it is pointed out in \cite{BE18,BJ18,Bo18} that $\mathcal R^k_{\omega}$, the space of finite dimensional geodesic rays associated to the space of Hermitian metrics $\mathcal H^k_{\omega}$ on $H^0(X,L^k)$ admits a natural metric $d_{p,k}^{c}$, likely representing the finite dimensional analog of our $d_p^c$ metrics. If one could show in the spirit of \cite[Theorem 1.1]{DLR18} that the metric spaces $(\mathcal R^k_{\omega},d_{p,k}^{c})$ approximate $(\mathcal R^p_{\omega},d_p^c)$ (or relevant parts of it) in the large $k$-limit, then that would open the door for a version of Theorem \ref{thm: approximation_dp_intr}, where the rays from $\mathcal R^{1,\bar 1}_\omega$ are replaced by rays induced by test configurations. Even if successful, it is not clear how convergence of the radial K-energy can be achieved (see \cite[Conjecture 2.5]{BJ18}), and for the difficulties that need to be overcome in this approach we refer to the comments following \cite[Conjecture 4.9]{Bo18}.

Further connections with geodesic rays are explored in \cite{DDL5}, related to  the metric geometry of the space of singularity types, and complex Monge--Amp\`ere equations with prescribed singularity. \vspace{-0.3cm}

\paragraph{Organization of the paper.} In Section 2 we recall basic facts about the $L^p$ Mabuchi geometry of the space of K\"ahler metrics, the relative Kolodziej type estimate of \cite{DDL4}, and we prove weighted versions of the classical inequalities of Clarkson and Ball--Carlen--Lieb that will be needed later. In Section 3 we prove Theorems \ref{thm: uniform_convex_intr} and \ref{thm: unique_geod_intr}  regarding uniform convexity, and uniqueness of geodesics in $L^p$ Mabuchi geometry when $p>1$. In Section 4 we study the chordal $L^p$ metric structures on the space of geodesic rays and prove Theorem \ref{thm: R_p_complete_geod_K_convex_intr}. In Section 5 we prove Theorem \ref{thm: approximation_dp_intr}, our main approximation result, and in Section 6 we show that the $C^{1, \bar 1}$ version of the uniform geodesic stability conjecture holds. \vspace{-0.6cm}

\paragraph{Acknowledgements.} We would like  to thank  Martin Kell and Alexander Lytchak for enlightening discussions on various notions of convexity in metric spaces.
We also thank L\'aszl\'o Lempert for his suggestions that improved the presentation of the paper. This work was partially supported by NSF grant 1610202. 

\section{Preliminaries} 

\subsection[The Lp Finsler geometry of the space of Kahler potentials]{The $L^p$ Finsler geometry of the space of K\"ahler potentials}
 \label{subsect: Lp finsler}
In this short section we recall the basics of finite energy pluripotential theory, as introduced by Guedj-Zeriahi \cite{GZ07}, and the Finsler geometry of the space of K\"ahler potentials, as introduced by the first author \cite{Da15}. For a detailed account on these matters we refer to the recent textbook \cite{GZ17} and lecture notes \cite{Da18}.

As a matter of convention for the duration of the paper we denote by $V$ the total volume of the K\"ahler class $[\omega]$:
$$V:= \int_X \omega^n.$$
By $\textup{PSH}(X,\omega)$ we denote the space of $\omega$-plurisubharmonic ($\omega$-psh) functions. Extending the ideas of Bedford--Taylor, Guedj--Zeriahi introduced the non-pluripolar Monge-Amp\`ere mass for a general potential $u \in \textup{PSH}(X,\omega)$ as the following limit \cite{GZ07}:
$$\omega_u^n := \lim_{k \to \infty}\mathbbm{1}_{\{u > -k\}} (\omega + \i\ddbar \max(u,-k))^n.$$
For such measures one has an estimate on the total mass $\int_X \o_u^n \leq \int_X \o^n=V$, and $\mathcal E_\o$ is the set of potentials with full/maximum mass:
$\mathcal E_\o := \{ u \in \textup{PSH}(X,\omega) \ \textup{s.t.} \ \int_X \o_u^n = \int_X \o^n=V\}.$
Furthermore, potentials $u \in \mathcal E_\o$ that satisfy an $L^p$ type integral condition are members of the so called \emph{finite-energy spaces} of \cite{GZ07}:
$$
\mathcal E^p_\o = \Big\{ u \in \mathcal{E}_{\omega} \ \textup{s.t.} \  \int_X |u|^p \o_u^n < +\infty\Big\}.
$$

Now we recall some of the main points on the $L^p$ Finsler geometry of the space of K\"ahler potentials. By definition, the space of K\"ahler potentials $\mathcal H_\o$ is an open convex subset of $C^\infty(X)$, hence one can think of it as a trivial Fr\'echet manifold. As a result, one can introduce on $\mathcal H_\o$ a collection of $L^p$ type Finsler metrics. If $u \in \mathcal H_\o$ and $\xi \in T_u \mathcal H_\o \simeq C^\infty(X)$, then the $L^p$ norm of $\xi$ is given by the following expression:
$$
\| \xi\|_{p,u} = \bigg(\frac{1}{V}\int_X |\xi|^p \o_u^n\bigg)^{\frac{1}{p}}.
$$
In case $p=2$, this construction reduces to the Riemannian geometry of Mabuchi \cite{Ma87}
(independently discovered by  Semmes \cite{Se92} and Donaldson \cite{Do99}).

Using these Finsler metrics, one can introduce path length metric structures $(\mathcal H_\o,d_p)$. In \cite[Theorem 2]{Da15}, the first author identified the  completion of these spaces with $\mathcal E^p_\o \subset \textup{PSH}(X,\o)$ from above, and it turns out that $(\mathcal E^p_\o,d_p)$ is a complete geodesic metric space. 

The distinguished $d_p$-geodesic segments of the completion $(\mathcal E^p_\o,d_p)$ are constructed as upper envelopes of quasi-psh functions, as we now elaborate. Let $S = \{0 < \textup{Re }s < 1 \} \subset \mathbb{C}$ be the unit strip, and $\pi_{S \times X}: S \times X \to X$ denotes projection to the second component.

We consider $u_0,u_1 \in \mathcal E^p_\o$. We say that the curve $[0,1] \ni t \to v_t \in \mathcal E^p_\o$ is a \emph{weak subgeodesic} connecting $u_0,u_1$ if $d_p(v_t,u_{0,1}) \to 0$ as $t \to 0,1$, and the extension $v(s,x) = v_{\textup{Re }s}(x)$ is $\pi^* \o$-psh on $S \times X$, i.e.,
$$\pi^* \o + i\partial_{S \times X} \bar \partial_{S \times X} v \geq 0, \ \textup{ as currents on } \ S \times X.$$
As shown in \cite{Da17,Da15}, a distinguished $d_p$-geodesic $[0,1] \ni t \to u_t 
\in \mathcal 
E^p_\o$ connecting $u_0,u_1$ can be obtained as the supremum of all weak subgeodesics:
\begin{equation}
\label{fegeod}
u_t := \sup \{v_t \ | \ t \to v_t  \textup{ is a subgeodesic connecting } u_0,u_1\}, \ t \in [0,1].
\end{equation}
Given $u_0,u_1\in \mathcal E^p_\o$, we call \eqref{fegeod}
the $L^p$ \emph{finite energy geodesic} (or simply finite energy geodesic) connecting $u_0,u_1$. 
Due to this ``Perron type" definition, finite energy geodesic segments satisfy a comparison principle.

In case the endpoints $u_0,u_1$ are from $\mathcal H_\o$,  the finite energy geodesic connecting them is actually $C^{1, \bar 1}$ on $\overline{S} \times X$, as shown by 
Chen \cite{Ch00} (for a survey see B\l ocki \cite{Bl13}, with the optimal result due to Chu--Tosatti--Weinkove \cite{CTW17}). 

Regarding the metric $d_p$ the following double estimate holds for some dimensional constant $C>1$  and all $p \geq 1$ \cite[Theorem 3]{Da15}:
$$\frac{1}{C}d_p(u_0,u_1)^p \leq  \frac{1}{V}\int_X |u_0 - u_1|^p \o_{u_0}^n + \frac{1}{V}\int_X |u_0 - u_1|^p \o_{u_1}^n \leq C d_p(u_0,u_1)^p, \ \ \ u_0,u_1 \in \mathcal E^p_\o.$$

We recall that for any $u \in \textup{PSH}(X,\omega)$ there exists $u_j \in \mathcal H_\o$ such that $u_j$ decreases to $u$. This is a result due to Demailly \cite{De92} with a simpler proof due to B{\l}ocki--Ko{\l}odziej \cite{BK07}. It is well known that the Monge--Amp\`ere energy $I: \mathcal E^1_\o \to \Bbb R$ defined by
$$I(u) = \frac{1}{V(n+1)} \sum_{j=0}^n \int_X u \omega^{n-j} \wedge \omega_u^j$$
is affine along finite energy geodesics \cite{Da15}. Moreover, the same is true for $\sup_X u_t$ in case $u_0=0$:
\begin{lemma}\label{lem: sup_X_linear} Let $[0,1] \ni t \to u_t \in \mathcal E^1_\o$ be a finite energy geodesic with $u_0 =0$. Then $t \to \sup_X u_t$ is affine.
\end{lemma}

This is essentially \cite[Theorem 1]{Da13}(ii), that is stated for bounded geodesics. Since finite energy geodesic segments can be approximated decreasingly by bounded geodesic segments, the above result follows as a consequence of Hartogs' lemma \cite[Proposition 8.4]{GZ17}.
For more on $L^p$ Mabuchi geometry we refer to  \cite[Chapter 3]{Da18}.

\subsection{The relative Ko{\l}odziej type estimate}

In this short subsection we recall the basics of relative pluripotential theory that are needed to state the relative Ko{\l}odziej type estimates of \cite{DDL4}. For more details we refer to the sequence of papers \cite{DDL1,DDL2,DDL3,DDL4}.

Let $E$ be a Borel subset of $X$. Given $\chi \in \textup{PSH}(X,\omega)$, we define the $\chi$-relative capacity of $E$ as 
\begin{equation}\label{eq: capphi_def}
 \textup{Cap}_\chi(E) := \sup \left \{\int_E \omega_{u}^n \setdef u \in \PSH(X,\omega),
  \ \chi-1\leq u \leq \chi \right \}. 
\end{equation}

When $\chi =0,$ we recover the classical Monge--Amp\`ere capacity $\textup{Cap}_\omega$ (see e.g. \cite{GZ05}). For more on this concept we refer to \cite[Section 4]{DDL4}.

Given $u\in \PSH(X,\omega)$, we recall the definition of envelopes with respect to singularity type, introduced  by Ross and Witt Nystr\"om \cite{RWN14}: 
$$
P[u]:= \textup{usc}\Big(\lim_{C\to +\infty} P(0,u+C)\Big) \in \textup{PSH}(X,\omega),
$$
where $P(\phi,\psi) := \sup\{v \in \textup{PSH}(X,\omega) \ \textup{s.t. } v \leq \phi \textup{ and } v \leq \psi\}.$
In addition to appearing in the statement of the relative Ko{\l}odziej type estimate below, this concept also plays a role in Theorem \ref{thm: approximation of d1 rays}, where it is used to approximate geodesic rays, via \cite{RWN14}.

Finally we recall the following $L^{\infty}$ estimate from \cite{DDL4}:

\begin{theorem}\textup{\cite[Theorem 3.3]{DDL4}}\label{thm: uniform_estimate_Kolodziej} 
	Let $a\in [0,1),A>0$, $\chi \in \PSH(X,\theta)$ and  $0\leq f \in L^p(X,\omega^n)$ for some $p>1$. Assume that  $u\in \PSH(X,\theta)$, normalized by $\sup_X u=0$, satisfies 
	\begin{equation}
		\label{eq: volume cap domination 0}
		\theta_u^n \leq f\omega^n + a\theta_{\chi}^n.
	\end{equation}
	Assume also that  
	\begin{equation}
		\label{eq: volume cap domination}
		\int_E f\omega^n \leq A [\textup{Cap}_{\chi}(E)]^2,
	\end{equation}
	for every Borel subset $E\subset X$. If $P[u]$ is less singular than $\chi$ then 
$$\chi -\sup_X \chi- C\Big(\|f\|_{L^p},p,(1-a)^{-1},A\Big) \leq u.$$
\end{theorem}

Here, given two potentials $u,v \in \PSH(X,\omega)$, we say that $u$ is less singular than $v$ if $u\geq v-C$, for some constant $C$. 

This theorem generalizes the classical estimates of Ko{\l}odziej from \cite{Ko98}, and it is used in \cite{DDL4} to solve complex Monge--Amp\`ere equations with prescribed singularity type, and to resolve the log-concavity conjecture of the volume in pluripotential theory. Here we will use it in Section 5 to show that it is possible to approximate $L^p$ geodesic rays with bounded ones that have converging radial K-energy.

\subsection{Weighted Clarkson and Ball--Carlen--Lieb type inequalities}

In this short preliminary section we point out relevant extensions of well known inequalities due to Clarkson \cite{Cl36} and Ball--Carlen--Lieb \cite{BCL94} for $L^p$ spaces, introducing a weight $\lambda \in [0,1]$ into these results. These theorems are almost certainly well known to experts in analysis, but we could not find the versions below in the literature.

\begin{theorem}\label{thm: Clarkson_lambda} Suppose that $p \geq 2, \ \lambda \in [0,1]$ and $f,g \in L^p(\nu)$, where $\nu$ is a measure on the set $X$. Then
\begin{equation}\label{eq: Clarkson_lambda}
\lambda \| f\|^p_p + (1-\lambda) \| g\|_p^p
 \geq \| \lambda f + (1-\lambda) g \|^p_p + \lambda^{\frac{p}{2}} (1-\lambda)^{\frac{p}{2}} \| f-g\|^p_p.
\end{equation}
\end{theorem}
\begin{proof}
Since $t \to |t|^{\frac{p}{2}}$ is a convex function, we can write the following estimates:
\begin{flalign*}
\lambda \| f\|^p_p + (1-\lambda) \| g\|_p^p
 & \geq \int_X (\lambda f^2 + (1-\lambda)g ^2 )^\frac{p}{2} d \nu \\
  & = \int_X ((\lambda f + (1-\lambda)g)^2 + \lambda (1-\lambda) (f-g)^2)^\frac{p}{2} d \nu \\
  & \geq \int_X (|\lambda f + (1-\lambda)g|^p + \lambda^{\frac{p}{2}} (1-\lambda)^{\frac{p}{2}} |f-g|^p) d \nu,
\end{flalign*}
where in the last step we have used that $(a^2 + b^2)^{\frac{1}{2}} \geq (a^p + b^p)^\frac{1}{p}, \ a,b \geq 0$.
\end{proof}

\begin{theorem}\label{thm: BCL_lambda} Suppose that $1 < p \leq 2, \ \lambda \in [0,1]$ and $f,g \in L^p(\nu)$, where $\nu$ is a measure on the set $X$. Then
\begin{equation}\label{eq: BCL_lambda}
\lambda \| f\|^2_p + (1-\lambda) \| g\|_p^2
 \geq \| \lambda f + (1-\lambda) g \|^2_p + (p-1)\lambda(1-\lambda) \| f-g\|^2_p.
\end{equation}
\end{theorem}
\begin{proof} The proof will be given using diadic approximation. Indeed, it is enough to prove \eqref{eq: BCL_lambda} for $\lambda = \frac{k}{2^m}, \ k,m \in \mathbb{N}$ with $1 \leq k \leq 2^m$. We will argue by induction on $m$. For $m=1$ and $k=0,1,2$, the statement of \eqref{eq: BCL_lambda} is either a triviality or reduces to \cite[Proposition 3]{BCL94}. Let us assume that $m >1$ and the statement holds for $m-1$. We can assume that $k$ is odd, as otherwise the inequality reduces to the case $m-1$. Using \cite[Proposition 3]{BCL94}, we start with the following estimate:
\begin{flalign}\label{eq: half_est}
\frac{1}{2} \Big\| \frac{k-1}{2^m} f + \Big(1-\frac{k-1}{2^m}\Big) g \Big\|_p^2 & + \frac{1}{2}  \Big\| \frac{k+1}{2^m} f + \Big(1-\frac{k+1}{2^m}\Big) g \Big\|_p^2 \geq \\
&\geq  \Big\| \frac{k}{2^m} f + \Big(1-\frac{k}{2^m}\Big) g \Big\|_p^2 + (p-1)\Big\| \frac{f}{2^m}-\frac{g}{2^m} \Big\|_p^2. \nonumber
\end{flalign}
Since both $k+1$ and $k-1$ are even, by the inductive step we also have that: 
\begin{flalign}\label{eq: k+1_est}
\frac{k+1}{2^m} \| f  \|_p^2 &+ \Big(1 -\frac{k+1}{2^m}\Big) \|g \|_p^2 \geq \\
& \geq  \Big\| \frac{k+1}{2^m} f + \Big(1-\frac{k+1}{2^m}\Big) g \Big\|_p^2 + (p-1)\frac{k+1}{2^m}\Big(1-\frac{k+1}{2^m}\Big) \| f-g \|_p^2, \nonumber
\end{flalign}
\begin{flalign}\label{eq: k-1_est}
\frac{k-1}{2^m} \| f  \|_p^2 &+ \Big(1 -\frac{k-1}{2^m}\Big) \|g \|_p^2 \geq \\
& \geq  \Big\| \frac{k-1}{2^m} f + \Big(1-\frac{k-1}{2^m}\Big) g \Big\|_p^2 + (p-1)\frac{k-1}{2^m}\Big(1-\frac{k-1}{2^m}\Big) \| f-g \|_p^2. \nonumber
\end{flalign}
Adding \eqref{eq: half_est}, \eqref{eq: k+1_est} and then using \eqref{eq: k-1_est} we arrive at 
\begin{equation*}
\frac{k}{2^m} \| f  \|_p^2 + \Big(1 -\frac{k}{2^m}\Big) \|g \|_p^2 \geq  \Big\| \frac{k}{2^m} f + \Big(1-\frac{k}{2^m}\Big) g \Big\|_p^2 + (p-1)\frac{k}{2^m}\Big(1-\frac{k}{2^m}\Big) \| f-g \|_p^2,
\end{equation*}
what we desired to prove.
\end{proof}

\begin{remark}As alluded to at the beginning of the subsection, in case $\lambda = \frac{1}{2}$, Theorem \ref{thm: Clarkson_lambda} and Theorem \ref{thm: BCL_lambda} recover the well known inequalities of  Clarkson \cite{Cl36} and Ball--Carlen--Lieb \cite[Proposition 3]{BCL94} respectively. 
\end{remark}

\section{Uniform convexity and uniqueness of geodesics}

Before proving the main result of this section, we first point out the following result about the ``spread" of geodesic segments in $\mathcal E^p_\o$, sharing a common smooth endpoint:
\begin{theorem}\label{thm: Lp_geod_spread} Suppose that $p \geq 1$, $u \in \mathcal H_\o$ and $[0,l] \ni t \to u_t,v_t \in \mathcal E^p_\o$ are two finite energy geodesic segments with $u=u_0=v_0$, and $l \in \mathbb{R}^+$. Then
\begin{equation}\label{eq: Lp_geod_spread}
\bigg[\int_X |\dot u_0 - \dot v_0|^p \o_u^n \bigg]^{\frac{1}{p}}\leq \frac{d_p(u_t,v_t)}{t}, \ \ t \in [0,l].
\end{equation}
\end{theorem}

\begin{proof}
We first assume that $u_l \geq v_l$. Furthermore, using $d_p$--approximation of the endpoints $u_l,v_l \in \mathcal E^p_\o$ by decreasing sequences of potentials in $\mathcal H_\o$, it is enough to prove \eqref{eq: Lp_geod_spread} for $C^{1,1}$-geodesics $t \to u_t,v_t$ with $u_l,v_l \in \mathcal H_\o$ (see \cite[Proposition 4.3]{BDL17}). 
 
Using the convexity condition \eqref{eq: ChCh3_metric_convex} and \cite[Lemma 5.1]{Da15} for $0\leq s\leq t \leq l$ we have 
	\begin{flalign*}
		 \frac{d_p(u_t,v_t)^p}{t^p} \geq \frac{d_p(u_s,v_s)^p}{s^p} & \geq \int_X \frac{(u_s-v_s)^p}{s^p}\omega_{u_s}^n. 
	\end{flalign*}
	As $s\to 0^+$, using the fact that the geodesics are $C^{1,1}$, we get that $(u_s-v_s)^p/s^p$ uniformly converges to $(\dot u_0 -\dot v_0)^p$ which is a continuous function on $X$. Since $\omega_{u_s}^n \to \omega_u^n$ weakly (see \cite[Theorem 5(i)]{Da15}) it follows  that 
	$$
	\frac{d_p(u_t,v_t)^p}{t^p}   \geq  \int_X (\dot u_0-\dot v_0)^p \omega_u^n. 
	$$
We now treat the general case, when $u_l$ and $v_l$ may not be comparable. By the previous step, for $t \in [0,l]$ we have  
  $$
  \frac{d_p(u_t,P(u_t,v_t))^p}{t^p}   \geq  \int_X |\dot u_0-\dot w_0^t|^p \omega_u^n\ \text{ and } \   \frac{d_p(v_t,P(u_t,v_t))^p}{t^p}   \geq  \int_X |\dot v_0-\dot w_0^t|^p \omega_u^n, 
  $$
  where $[0,t] \ni s \mapsto w_s^t\in \Ecp$ is the finite energy geodesic connecting $w_0:=u_0$ and $w_t:=P(u_t,v_t)$. 
  
Due to the comparison principle for geodesics, we note that $\dot w^t_0 \leq \dot u_0,\dot v_0$. Using the Pythagorean formula \cite[Corollary 4.14]{Da15}  and the inequality $a^p + b^p \geq \max(a^p,b^p) \geq |a-b|^p, \ a,b\geq 0,$ we can sum up the above inequalities to arrive at the conclusion:
\begin{flalign*}
\frac{d_p(u_t,v_t)^p}{t^p} = \frac{d_p(u_t,P(u_t,v_t))^p}{t^p} + \frac{d_p(v_t,P(u_t,v_t))^p}{t^p} \geq  \int_X |\dot u_0-\dot v_0|^p \omega_{u}^n, \ \ t \in [0,l].  
\end{flalign*}
\end{proof}

Before proceeding we note that Theorem \ref{thm: Lp_geod_spread} implies the following Lidskii type inequality proved in the case of Hodge type K\"ahler metrics in \cite{DLR18}: 

\begin{corollary}\label{cor: Lidskii} If $\alpha,\beta,\gamma \in \mathcal E^p_\o$ with $\alpha \geq \beta \geq \gamma$ then:
$$d_p(\beta,\gamma)^p \leq d_p(\alpha,\gamma)^p - d_p(\alpha,\beta)^p.$$
\end{corollary}
\begin{proof}By density it is enough to show this estimate for $\alpha,\beta,\gamma \in \mathcal H_\o$. Let $[0,1]\ni t \to u_t,v_t \in \mathcal E^p_\o$ be the increasing/decreasing $C^{1,1}$-geodesics joining $u_0: =\beta,u_1:=\alpha$ and $v_0: =\beta,v_1:=\gamma$ respectively. Then, due to $t$-monotonicity, Theorem \ref{thm: Lp_geod_spread}, and \cite[Theorem 1]{Da15}, the following holds:
$$d_p(\alpha,\gamma)^p = d_p(u_1,v_1)^p \geq \int_X |\dot u_0 - \dot v_0|^p \o^n_\beta \geq \int_X \left(|\dot u_0|^p + |\dot v_0|^p\right) \o^n_\beta = d_p(\alpha,\beta)^p + d_p(\gamma,\beta)^p.$$
\end{proof}

Next we prove the  main result of this section about the uniform convexity of the spaces $(\mathcal E^p_\o,d_p)$ for  $p > 1$. This will follow after an adequate combination of Theorem \ref{thm: Lp_geod_spread} and the extension of the inequalities of Clarkson and Ball--Carlen--Lieb, obtained in the previous section.

\begin{theorem} \label{thm: uniform_convex}Suppose that $u \in \mathcal E^p_\o, \ \lambda \in [0,1]$ and $[0,1] \ni t \to v_t \in \mathcal E^p_\o$ is a finite energy geodesic segment. Then the following hold:\\
(i) $d_p(u,v_\lambda )^2 \leq (1-\lambda)d_p(u,v_{0})^2 +\lambda d_p(u,v_1)^2 - (p-1)\lambda(1-\lambda )d_p(v_0,v_1)^2$, if $1 < p \leq 2.$\\
(ii) $d_p(u,v_\lambda)^p \leq (1-\lambda )d_p(u,v_{0})^p +\lambda d_p(u,v_1)^p - \lambda^{\frac{p}{2}} (1-\lambda)^{\frac{p}{2}}d_p(v_0,v_1)^p$, if $2 \leq  p.$
\end{theorem}

\begin{proof} To begin, let $p \geq 1$ and $\lambda \in [0,1]$. By density (and \cite[Proposition 4.3]{BDL17}) we can assume that $u,v_0,v_1 \in \mathcal H_\o$ and hence $t \to v_t$ is $C^{1,\bar{1}}$.  

Fixing $\varepsilon >0$ momentarily, let $[0,1] \ni t \to v^\varepsilon_t \in \mathcal H_\o$ be Chen's smooth $\varepsilon$-geodesic connecting $v_0,v_1 \in \mathcal H_\o$ (\cite{Ch00}, for a survey see \cite[Section 3.1]{Da18}). Moreover, let $[0,1] \ni t \to \alpha_t^{\lambda,\varepsilon} \in \mathcal E^p_\o$ be the $C^{1,1}$ geodesic connecting $u$ and $v_\lambda^\varepsilon$. Let $[0,\lambda] \ni t \to h^\varepsilon_t \in \mathcal E^p_\o$ be the $C^{1,1}$ geodesic connecting $v_0$ and $v_\lambda^\varepsilon$. Similarly, let $[\lambda,1] \ni t \to k^\varepsilon_t \in \mathcal E^p_\o$ be the $C^{1,1}$ geodesic connecting  $v_\lambda^\varepsilon$ and $v_1$.

We now assume that $2 \leq p$ to address (ii).  Using Theorem \ref{thm: Lp_geod_spread} twice, for pairs of geodesics emanating from $v^\varepsilon_\lambda$, we conclude that
$$\int_X \big|\dot \alpha^{\lambda,\varepsilon}_1+ \lambda \dot h^\varepsilon_\lambda\big|^p \o^n_{v^\varepsilon_\lambda}\leq d_p(u,v_0)^p, \ \ \ \int_X \big|\dot \alpha^{\lambda,\varepsilon}_1- (1-\lambda) \dot k^\varepsilon_\lambda\big|^p \o^n_{v^\varepsilon_\lambda}\leq d_p(u,v_1)^p.$$
By the comparison principle for geodesics, we have that $v^\varepsilon_t \leq h^\varepsilon_t \leq v_t, \ t \in [0,\lambda]$ and $v^\varepsilon_t \leq k^\varepsilon_t \leq v_t, \ t \in [\lambda,1]$. Again, by the comparison principle, the concatenation of $t \to h^\varepsilon_t$ and $t \to k^\varepsilon_t$ is $t$-convex and we obtain that $\dot h^\varepsilon_\lambda \to \dot v_\lambda$ and $\dot k^\varepsilon_\lambda \to \dot v_\lambda$ uniformly on $X$. 
Using this and the above two estimates we can write:
\begin{flalign}\label{eq: L_p_geq_2_ineq}
(1-\lambda) d_p(u, & v_0)^p + \lambda  d_p(u,v_1)^p  \geq \int_X  (1-\lambda)|\dot \alpha^{\lambda,\varepsilon}_1+ \lambda \dot h^\varepsilon_\lambda|^p + \lambda |\dot \alpha^{\lambda,\varepsilon}_1- (1-\lambda) \dot k^\varepsilon_\lambda|^p \o^n_{v^\varepsilon_\lambda} \nonumber \\
&\geq  (1-\lambda) \int_X |\dot \alpha^{\lambda,\varepsilon}_1+  \lambda \dot v_\lambda|^p \o^n_{v^\varepsilon_\lambda} + \lambda \int_X |\dot \alpha^{\lambda,\varepsilon}_1- (1-\lambda) \dot v_\lambda|^p  \o^n_{v^\varepsilon_\lambda} - O(\varepsilon) \nonumber \\
&\geq  \int_X |\dot \alpha_1^{\lambda,\varepsilon}|^p\o^n_{v^\varepsilon_\lambda} + \lambda^{\frac{p}{2}}(1-\lambda)^{\frac{p}{2}} \int_X |\dot v_\lambda |^p \o^n_{v^\varepsilon_\lambda}- O(\varepsilon) \nonumber\\
&= d_p(u,v^\varepsilon_\lambda)^p + \lambda^{\frac{p}{2}}(1-\lambda)^{\frac{p}{2}}\int_X |\dot v_\lambda |^p \o^n_{v^\varepsilon_\lambda}- O(\varepsilon),
\end{flalign}
where  in the third line we have used Theorem \ref{thm: Clarkson_lambda}, and in the last line we have used \cite[Theorem 1]{Da15}. Letting $\varepsilon \to 0$, since ${\omega}_{v^\varepsilon_\lambda}^n \rightharpoonup {\omega}^n_{v_\lambda}$ and $O(\varepsilon) \to 0 $, another application of \cite[Theorem 1]{Da15} gives $(ii)$.

Now we assume that $1 < p \leq 2$ and we address the inequality of $(i)$. The proof is exactly the same, except for  \eqref{eq: L_p_geq_2_ineq}, where we use the estimate of Theorem \ref{thm: BCL_lambda} instead of Theorem \ref{thm: Clarkson_lambda}.
\end{proof}
\begin{remark} Suppose that $\omega$ is the curvature of a Hermitian line bundle $(L,h)$. By exactly the same arguments, one can show that the inequalities of Theorem \ref{thm: uniform_convex} also hold for the finite dimensional $L^p$ type metric spaces $(\mathcal H^k_\omega,d_{p,k})$, as considered in \cite{DLR18}. Using the quantization scheme of this paper \cite[Theorem 1.2]{DLR18}, an alternative proof of Theorem \ref{thm: uniform_convex} can be thus given  when $[\o]$ is integral. 
\end{remark}

Finally we point out that using the above result one can show that the finite energy geodesic segments of $\mathcal E^p_\o$ are the only metric geodesics when $p>1$:
\begin{theorem} Let $p \in (1,\infty)$, and suppose that $[0,1] \ni t \to v_t \in \mathcal E^p_\o$ is the finite energy geodesic connecting $v_0,v_1 \in \mathcal E^p_\o$. Then $t \to v_t$ is the only $d_p$-geodesic connecting $v_0,v_1$.
\end{theorem}
\begin{proof}Suppose that $[0,1] \ni t \to u_t \in \mathcal E^p_\o$ is a $d_p$-geodesic connecting $v_0,v_1$, and let $h_t \in \mathcal E^p_\o$ be the $d_p$-midpoint of the finite energy geodesic connecting $u_t,v_t, \ t \in [0,1]$. Assuming that $u_t \neq v_t$, Theorem \ref{thm: uniform_convex} implies that $d_p(v_0,h_t) < \max\{d_p(v_0,u_t),d_p(v_0,v_t)\}=t d_p(v_0,v_1)$. Similarly, $d_p(v_1,h_t) < \max\{d_p(v_1,u_t),d_p(v_1,v_t)\}=(1-t) d_p(v_0,v_1)$. The triangle inequality now gives a contradiction, implying that $u_t = v_t, \ t \in [0,1]$.
\end{proof}
A more careful analysis of the above proof yields the following:
\begin{proposition} \label{prop: endpoint_distance_stability} Suppose that $p >1$ and $[0,1] \ni l \to u_l \in \mathcal E^p_\o$ is a finite energy geodesic. Let $v \in \mathcal E^p_\o$ such that $d_p(v,u_0) \leq (t+\varepsilon) d_p(u_0,u_1)$ and $d_p(v,u_1) \leq (1 - t+\varepsilon) d_p(u_0,u_1)$ for some $\varepsilon >0$ and $t \in [0,1]$. Then there exists $C(p) >0$ such that 
$$d_p(v,u_t) \leq \varepsilon^\frac{1}{r}  C d_p(u_0,u_1),$$
where $r:=\max(2,p)$.
\end{proposition}
\begin{proof} Let $h$ be the $d_p$-midpoint of the finite energy geodesic connecting $v$ and $u_t$. Then Theorem \ref{thm: uniform_convex} implies that  
$$d_p(u_0,h) \leq \Big[\frac{1}{2} d_p(u_0,v)^r +\frac{1}{2}d_p(u_0,u_t)^r - c d_p(v,u_t)^r \Big]^{\frac{1}{r}},$$
$$d_p(u_1,h) \leq \Big[\frac{1}{2} d_p(u_1,v)^r +\frac{1}{2}d_p(u_1,u_t)^r - c d_p(v,u_t)^r\Big]^{\frac{1}{r}} ,$$
for $r:=\max(p,2)$, and $c:=c(p) \in (0,1)$. Adding these estimates and using the triangle inequality we arrive at:
\begin{flalign*}
d_p(u_0,u_1) \leq \Big[(t+\varepsilon)^rd_p(u_0,u_1)^r - c d_p(v,u_t)^r \Big]^{\frac{1}{r}}+ \Big[(1-t+\varepsilon)^r d_p(u_0,u_1)^r - c d_p(v,u_t)^r\Big]^{\frac{1}{r}}
\end{flalign*}
After dividing by $d_p(u_0,u_1)$, basic calculus yields that 
$$\frac{d_p(v,u_t)^r}{d_p(u_0,u_1)^r} \leq \max \bigg(\frac{(t + \varepsilon)^r - t^r}{c}, \frac{(1-t + \varepsilon)^r - (1-t)^r}{c}\bigg),$$
implying that 
$d_p(v,u_t) \leq \varepsilon^\frac{1}{r} C d_p(u_0,u_1)$, as desired.
\end{proof}

\section[The metric geometry of weak Lp geodesic rays]{The metric geometry of weak $L^p$ geodesic rays}

For $u \in \mathcal E^p_\o$ let $\mathcal R^p_u$ denote the space of finite energy $L^p$ geodesic rays emanating from $u$. Note that we don't assume that the rays are unit speed, or even non-constant. 

Following terminology from metric space theory \cite{BH99}, two rays $\{u_t\}_t,\{v_t\}_t$ are parallel if $d_p(u_t,v_t)$ is uniformly bounded. Given the characteristics of the finite energy spaces, any ray admits a unique parallel ray emanating from an outside point, thus the $d_p$-geometries  verify Euclid's 5th postulate for half-lines, answering an open question raised in \cite[Remark 1.6]{ChCh3}:

\begin{proposition}\label{prop: parallel}
Let $u,v \in \mathcal E^p_\o$ then for any $\{u_t\}_t \in \mathcal R^p_u$ there exists a unique $\{v_t\}_t \in \mathcal R^p_v$ such that $\{u_t\}_t$ is parallel to $\{v_t\}_t$, giving a bijection $\mathcal P_{uv}: \mathcal R^p_u \to  \mathcal R^p_v$. Moreover $d_p(u_t,v_t) \leq d_p(u,v), \ t \geq 0$.
\end{proposition}

\begin{proof} 
Uniqueness follows from $d_p(u_t,v_t) \leq d_p(u,v), \ t \geq 0$, which is a simple consequence of the convexity of $t \to d_p(u_t,v_t)$ \eqref{eq: ChCh3_metric_convex}.

We first argue the proposition for $u \geq v$, using the maximum principle. Consider the finite energy geodesic segments $[0,t] \ni l \to v^t_l \in \mathcal E^p_\o$, with $v^t_0 = v$ and $v^t_t = u_t$. Then by the comparison principle for geodesics we get that, for $0\leq l \leq t \leq t'$,  $v^{t'}_{t} \leq u_t = v^{t}_t$, hence $v^{t'}_l \leq v^{t}_l$. Also, \eqref{eq: ChCh3_metric_convex} implies
$$\frac{d_p(v^t_l , u_l)}{t-l} \leq \frac{d_p(v,u)}{t}.$$
Putting the last two sentences together, \cite[Proposition 4.3]{BDL17} implies that 
$l \to v_l := \lim_{t \to \infty} v^t_l \in \mathcal E^p_\o$ is a finite energy geodesic ray such that $d_p(u_l,v_l) \leq d_p(u,v), \ l \geq 0$.

If $u \leq v$, the proposition holds by the same argument (the inequality $v^t_l \leq v^{t'}_l$ being  the only difference). 

To treat the general case, we simply notice that $h := \max (\sup_X u,\sup_X v) \in \mathcal H_\o \subset \mathcal E^p_\o$ and $h \geq u,v$. This allows to introduce a ray $\{h_t\}_t \in \mathcal R^p_h$ such that $d_p(u_t,h_t) \leq d_p(u,h)$. Since $h \geq v$, it is now possible to introduce another ray $\{v_t\}_t \in \mathcal R^p_v$ with $d_p(v_t,h_t) \leq d_p(v,h)$. The estimate $d_p(u_t,v_t) \leq d_p(u,h) + d_p(v,h)$, now follows from the triangle inequality.
\end{proof}

Next we introduce the chordal metric on $\mathcal R^p_u$:
\begin{equation}\label{eq: chordal_metric_def}
d_{u,p}^c(\{u_t\}_t,\{v_t\}_t):= \lim_{t \to \infty} \frac{d_p(u_t,v_t)}{t}, \ \ \{ u_t\}_t,\{ v_t\}_t \in \mathcal R^p_u.
\end{equation}
That the above increasing limit exists and is finite follows again from \eqref{eq: ChCh3_metric_convex} and the triangle inequality. 
As we now clarify,  $(\mathcal R^p_\o,d_p^c)$ is in fact a complete geodesic metric space.

\begin{theorem}\label{thm: Rp is complete} For any $u \in \mathcal E^p_{\omega}, \ p \geq 1$, $(\mathcal R^p_u,d^c_{u,p})$ is a complete metric space. Moreover for any $v \in \mathcal E^p_{\omega}$ the map $\mathcal P_{uv}: (\mathcal R^p_u,d^c_{u,p}) \to (\mathcal R^p_v,d^c_{v,p})$ is an isometry. 
\end{theorem}

Some aspects of the proof below can be traced back to \cite[Lemma 3.1]{BDL16}.

\begin{proof}That $d^c_{p,u}$ satisfies the triangle inequality follows from the triangle inequality of $d_p$. To argue non-degeneracy, suppose that $d^c_{p,u}(\{u_t\}_t,\{v_t\}_t)=0$. This implies that the increasing function $f(t) = d_p(u_t,v_t)/t$ satisfies $f(0)=0$ and $\lim_{t \to \infty}f(t)=0$. Consequently $f(t)=0, \ t \geq 0$, implying that $u_t=v_t, \ t \geq 0$.

Now suppose that $\{ u^j_t\}_t \subset \mathcal R^p_u$ is a $d^c_{u,p}$-Cauchy sequence. Fixing $l > 0$ we have that
\begin{equation}\label{eq: Cauchy_ray}
\frac{d_p(u^j_l, u^k_l)}{l} \leq d^c_{u,p}(\{u^j_t\}_t,\{u^k_t\}_t).
\end{equation}
Consequently $\{u^j_l\}_j \subset \mathcal E^p_\o$ is a $d_p$-Cauchy sequence with limit $u_l \in \mathcal E^p_\o$. By the endpoint stability of geodesic segments in $\mathcal E^p_\o$ (\cite[Proposition 4.3]{BDL17}) it follows that $t \to u_t$ is a geodesic ray. More importantly, letting $k \to \infty$ in \eqref{eq: Cauchy_ray} it follows that  $\frac{d_p(u^j_l, u_l)}{l}$ is arbitrarily small for high enough $j$ and any $l > 0$. This in turn implies that $d^c_{u,p}(\{u^j_t\}_t,\{u_t\}_t) \to 0$, giving completeness.

That the map $\mathcal P_{uv}$ is an isometry, follows from the definition of parallel geodesic rays and the triangle inequality for $d_p$.
\end{proof}

By this theorem, no extra information is gained by choice of initial metric, hence going forward we will only consider the space $(\mathcal R^p_{\omega}, d_p^c)$, the collection of rays emanating from $0 \in \mathcal H_{\omega} \subset \mathcal E^p_{\omega}$. 

\paragraph{Approximation of finite energy rays. } In this paragraph we point out that bounded geodesic rays (running inside $\textup{PSH}(X,\omega) \cap L^\infty$) are dense among the rays of $\mathcal R^p_\o$. Later, in the presence of finite radial K-energy we will sharpen this result further.

First we start with an auxilliary result, which is a consequence of Corollary \ref{cor: Lidskii}, and it is the radial analog of \cite[Lemma 4.16]{Da15}: 

\begin{lemma}\label{lem: decreasing_ray_d_plimit} Let $\{ u_t\}_t,\{u^j_t\} \in \mathcal R^p_\o$ such that  $u^j_t$ is decreasing (increasing a.e.) to $u_t$ as $j \to \infty$ for all $t \geq 0$. Then, $d_p^c(\{u_t^j\}_t,\{u_t\}_t) \to 0$. 
\end{lemma}

\begin{proof} We start by noticing that $t \to \sup_X u_t$ and $t \to \sup_X u_t^j$ are linear (Lemma \ref{lem: sup_X_linear}). By our assumption we have that $\sup_X u^j_1 \to  \sup_X u_1$ \cite[Proposition 8.4]{GZ17}, hence after possibly subtracting the same $t$-linear term from all our rays, without loss of generality we can assume that  $\sup_X u_t,\sup_X u_t^j \leq 0$. By convexity we will obtain that $0 \geq u^j_t \geq u_t$  ($0 \geq u_t \geq u^j_t$) for all $j$ and $t \geq 0$. Consequently, Corollary \ref{cor: Lidskii} is applicable to yield that:
\begin{equation}\label{eq: Lidskii_est}
\frac{d_p(u^j_t,u_t)^p}{t^p} \leq \frac{|d_p(0,u_t)^p-d_p(0,u_t^j)^p|}{t^p}={|d_p(0,u_1)^p-d_p(0,u_1^j)^p|}, \ \ t \geq 0,
\end{equation}
where we have used that $t \to d_p(0,u^j_t)$ and  $t \to d_p(0,u_t)$ are linear. Now \cite[Lemma 4.16]{Da15} gives that $d_p(u_1^j,u_1) \to0$, in particular $d_p(0,u_1^j) \to d_p(0,u_1)$, finishing the proof. 
\end{proof}

\begin{remark}\label{rem: lemma_conv_not_monotone} Analyzing the above argument we see that in Lemma \ref{lem: decreasing_ray_d_plimit} the conditions can be significantly weakened in some cases. For example, it is enough to assume that $u_t \leq u_t^j, \ t \geq 0, j \geq 0$, there exists $C>0$ such that $u_1^j \leq C, \ j \geq 0 $, and that $u_1^j$ converges to $u_1$ pointwise on $X$, with the exception of a pluripolar set. Using \cite[Lemma 5.1]{Da15} we  obtain that $d_p(u_1,u_1^j)^p \leq \int_X |u_1 - u_1^j|^p \omega_{u_1}^n,$ and the dominated convergence theorem allows to conclude that the right hand side of \eqref{eq: Lidskii_est} still converges to zero.
\end{remark}

\begin{theorem}\label{thm: approximation of d1 rays}
Let $\{ u_t\}_t \in \mathcal R^p_\o$. Then there exists a sequence $\{ u^j_t\}_t \in \mathcal R^p_\o$ such that $u_t^j \in \textup{PSH}(X,\omega) \cap L^\infty$  and $u^j_t \searrow u_t$ as $j \to \infty$ for all $t \geq 0$. In particular $d_p^c(\{u_t^j\}_t,\{u_t\}_t) \to 0$, and we can choose $\{u_t^j\}_t$ such that 
\begin{equation}\label{eq: C_0_bound_approx}
\max\Big(u_t, (\sup_X u_1-j) t\Big) \leq u_t^j  \leq t \sup_X u_1. 
\end{equation}
\end{theorem}

\begin{proof}
It follows from Lemma \ref{lem: sup_X_linear} that $t \to \sup_X u_t/t, \ t > 0$ is constant, hence we can assume (by adding $Ct$ to $u_t$) that $\sup_X u_t=0, \ t \geq 0$. Consequently $t \to u_t$ is $t$-decreasing. 
For $\tau \in \mathbb{R}$ and  $x\in X$ we introduce 
\begin{equation}\label{psi_def}
	\psi_{\tau}(x) := \inf_{t > 0} (u_t(x) -t\tau).
\end{equation}
From Kiselman's minimum principle \cite{Kis78} we have that $\psi_{\tau}\equiv -\infty$ or $\psi_{\tau} \in \PSH(X,\omega)$. More precisely, since $\sup_X u_t=0$ we have that $\psi_\tau \in \PSH(X,\omega)$ for $\tau \leq 0$, and $\psi_{\tau}\equiv -\infty$ for all $\tau >0$. Observe also that $\tau \to \psi_{\tau}$ is $\tau$-decreasing and $\tau$-concave. For all $x\in X$ with $\psi_0(x)>-\infty$ the curve $t \to u_t(x)$ is continuous in $(0,+\infty)$. Hence, by the involution property of the Legendre transform, for such $x$ we have
\begin{equation}\label{eq: u_t_def}
u_t(x)  = \sup_{\tau < 0} (\psi_{\tau}(x) + t\tau)= \sup_{\tau \in \Bbb R} (\psi_{\tau}(x) + t\tau), \ \ \ t > 0. 
\end{equation}
For $\varepsilon>0, \tau < 0$,  set 
	$$
	\psi^{\varepsilon}_{\tau}(x) := \max(0, 1+\varepsilon \tau) \psi_{\tau}, \ \text{and} \ \phi^{\varepsilon}_{\tau} := P[\psi^{\varepsilon}_{\tau}]. 
	$$
	We define $\phi^{\varepsilon}_0:= \lim_{\tau \to 0^-} \phi^{\varepsilon}_{\tau}$. 
	
Since $\tau \to \psi_{\tau}$ is $\tau$-concave, $\tau$-decreasing, and $\psi_{\tau}\leq 0$, it is elementary to see that $\tau \to \psi^{\varepsilon}_{\tau}$ is also $\tau$-concave and $\tau$-decreasing. By elementary properties of $P[\cdot]$ we get that $\tau \to \phi^{\varepsilon}_{\tau}$ is also $\tau$-concave and $\tau$-decreasing (see the proof of \cite[Proposition 4.6]{DDL3}).
As a consequence of a result due to Ross-Witt Nystr\"om \cite{RWN14} (further elaborated in \cite[Corollary 1.3]{DDL3}) the curve 
\begin{equation}\label{eq: u_t_def_eps}
[0,\infty) \ni t \to u^\varepsilon_t(x) := \sup_{\tau < 0}  (\phi^{\varepsilon}_{\tau}(x) + t\tau) \in \textup{PSH}(X,\omega) \cap L^\infty
\end{equation}
is a (bounded) geodesic ray emanating from $0$. 

We now prove that $u^{\varepsilon}_t \searrow u_t$ as $\varepsilon \searrow 0$, for any $t \geq 0$. For $t=0$ there is nothing to prove since $u^{\varepsilon}_0=u_0=0$ on $X$. Fix now $t>0$ and $x\in X$ with $\psi_0(x)>-\infty$. Then, using $\tau$-concavity, there exists $C>0$ depending on $\psi_0(x), t$ (but not on $\varepsilon$) such that 
$$
u_t^{\varepsilon}(x) = \sup_{-C\leq \tau\leq 0} (\phi_{\tau}^{\varepsilon}(x) + t\tau), \ \text{and} \ u_t(x) = \sup_{-C\leq \tau \leq 0} (\psi_{\tau}(x) +t\tau).
$$
By Lemma \ref{lem: DDL2} below, the family of functions $\tau \mapsto \phi_{\tau}^{\varepsilon}(x)$ decreases pointwise to the function $\tau \mapsto \psi_{\tau}(x)$ as $\varepsilon\to 0^+$ for $\tau <0$. Using $\tau$-concavity and the fact that $\psi_0(x)>-\infty$, one can extend this convergence to $\tau=0$ as well. 
Hence by Dini's theorem the convergence is
uniform on $[-C,0]$. It thus follows that $u_t^{\varepsilon}(x) \searrow u_t(x)$ as $\varepsilon\to 0^+$. We conclude that  $u^{\varepsilon}_t$ decreases to $u_t$ a.e. on $X$. But these are $\omega$-psh functions, so the convergence holds everywhere on $X$.

That $d_p^c(\{u^{\varepsilon}_t\}_t,\{u_t\}_t) \to 0$ as $\varepsilon\to 0^+$, simply follows from Lemma \ref{lem: decreasing_ray_d_plimit}. 
 
Since, $\phi^\varepsilon_\tau = 0$ for $\tau \leq -1/\varepsilon$ and $\psi_\tau \leq \phi^\varepsilon_\tau$, basic properties of  Legendre transforms imply that $u_t \leq u^{\varepsilon}_t \leq 0$ and $-\frac{t}{\varepsilon} \leq u^\varepsilon_t \leq 0$, since $\psi_\tau \leq \phi^\varepsilon_\tau$ for all $\tau$ and $\phi^\varepsilon_\tau =0$ for $\tau < -\frac{1}{\varepsilon}$. This immediately yields \eqref{eq: C_0_bound_approx} with $\varepsilon=1/j$.
\end{proof}
\begin{lemma}\label{lem: DDL2}
Assume that $\{u_t\}_t \in \mathcal R^1_\o$ satisfies $\sup_X u_t=0$ for all $t\geq 0$.  Then for $\psi_{\tau}$ defined in \eqref{psi_def} we have that $\int_X \omega^n_{\psi_\tau}>0$ for all $\tau<0$. Additionally for any $\tau < 0$,  
\begin{equation}\label{eq: ceiling_limit}
\lim_{\varepsilon \to 0} \phi^\varepsilon_\tau= \lim_{\varepsilon \to 0}P[(1+\varepsilon \tau )\psi_\tau]=\psi_\tau.
\end{equation}
\end{lemma}
\begin{proof}  By the involution property, application of the Legendre transform twice gives back the original convex function. In particular, we have that  $\sup_{\tau} \psi_\tau(x)=\lim_{\tau\to -\infty}\psi_{\tau}(x) = u_0(x)$ for all $x\in X$ such that $\lim_{t\to 0}u_t(x)=0$.   In particular, we get that $\psi_\tau$ increases a.e. to $0$ as $\tau \to -\infty$. According to \cite[Remark 2.5]{DDL2} we obtain that $\lim_{\tau \to -\infty} \int_X \omega_{\psi_\tau}^n = \int_X \omega^n >0$.

Fixing $\tau <0$, this last identity implies existence of $\tau_0 < \tau$ such that $\int_X \omega_{\psi_{\tau_0}}^n >0$.
By $\tau$-concavity of $\tau \to \psi_\tau$ we get that 
$$\psi_\tau \geq \frac{\tau}{\tau_0} \psi_{\tau_0} + \Big(1-\frac{\tau}{\tau_0}\Big) \psi_0.$$ 
Finally, by monotonicity  \cite[Theorem 1.2]{WN17} and the multi-linearity of the non-pluripolar mass we obtain that $\int_X \omega_{\psi_\tau}^n>0$, as desired.

To argue \eqref{eq: ceiling_limit}, we start by noting that $\lim_{\varepsilon \to 0}P[(1+\varepsilon \tau )\psi_\tau]\geq \psi_\tau$, and according to \cite[Theorem 1.2]{WN17} and \cite[Remark 2.5]{DDL2} we get that $$\int_X \omega_{\psi_\tau}^n \leq  \int_X \omega_{\lim_\varepsilon P[(1+\varepsilon \tau )\psi_\tau]}^n\leq \lim_{\varepsilon }\int_X \omega_{P[(1+\varepsilon \tau )\psi_\tau]}^n=\lim_{\varepsilon \to 0}\int_X \omega_{(1+\varepsilon \tau )\psi_\tau}^n=\int_X \omega_{\psi_\tau}^n.$$
Hence we have equality everywhere, and all the integrals are positive. Consequently, $\lim_{\varepsilon \to 0}P[(1+\varepsilon \tau )\psi_\tau] \in F_{\psi_\tau}$ with the notation of \cite[Theorem 3.12]{DDL2}.

It follows from \cite[Proposition 5.1]{Da13} (or \cite[Lemma 3.17]{DDL1}) that $P[\psi_{\tau}]=\psi_{\tau}$, for all $\tau\leq 0$ (the result in these works is only stated for rays of bounded potentials, however the proof only uses the comparison principle that holds for finite energy rays as well, implying the result for these more general rays). 
Putting everything together  \cite[Theorem 3.12]{DDL2} implies that $\lim_{\varepsilon \to 0}P[(1+\varepsilon \tau )\psi_\tau]=\psi_\tau$, as desired.
\end{proof}

\paragraph{The construction of geodesic segments in $\mathcal R^p_\o$. }Next we show that  points of $(\mathcal R^p_\o, d_p^c)$ can be connected by geodesic segments. We first treat the case $p>1$, where due to uniform convexity, the construction can be carried out directly. The case $p=1$ will be treated using approximation, via Theorem \ref{thm: approximation of d1 rays}.
\begin{theorem}\label{thm: chord p}
If $p>1$, then $(\mathcal R^p_{\omega}, d_p^c)$ is a complete geodesic metric space. 
\end{theorem}
The proof below shares similarities with the angle bisection techniques of \cite{KL97}.
\begin{proof} By Theorem \ref{thm: Rp is complete}, we only have to show that any two rays $\{u_t\}_t,\{v_t\}_t \in \mathcal R^p_{\omega}$ can be joined by a distinguished   $d^c_p$-geodesic when $p>1$.

For any $t \geq 0$, we denote by $[0,1] \ni \alpha \to h_{t,\alpha} \in \mathcal E^p_{\omega}$ the finite energy geodesic connecting $u_t$ and $v_t$. To avoid introducing further variables, by $[0,t] \ni s \to \frac{s}{t} h_{t,\alpha} \in \mathcal E^p_{\omega}$ we denote the finite energy geodesic connecting $0$ and $h_{t,\alpha}$. Finally, we can assume that $u_t \neq v_t$ for $t$ large enough. Indeed, if this does not hold, then \eqref{eq: ChCh3_metric_convex} would give that $\{u_t\}_t = \{v_t\}_t$ and the geodesic connecting the two rays is the constant one.

First we show that  for any $\alpha \in [0,1]$ and $l \geq 0$ there exists $w_{l,\alpha} \in \mathcal E^p_{\omega}$ such that $\lim_{t \to \infty} \frac{l}{t}h_{t,\alpha} = w_{l,\alpha}$. By endpoint stability of geodesic segments (\cite[Proposition 4.3]{BDL17}), this will automatically imply that $\{w_{t,\alpha}\}_t \in \mathcal R^p_{\omega}$. As we will see, $\alpha \to \{w_{t,\alpha}\}_t$ will represent the $d^c_p$-geodesic connecting $\{u_t\}_t$ and $\{v_t\}_t$.

Again, from \eqref{eq: ChCh3_metric_convex} it follows that for any $\alpha \in [0,1]$ and $0 < s \leq t$ we have
\begin{equation}\label{eq: alpha_ineq}
\frac{d_p(u_s,\frac{s}{t}h_{t,\alpha})}{s} \leq \frac{d_p(u_t,h_{t,\alpha})}{t}=\frac{\alpha d_p(u_t,v_t)}{t} \leq \alpha d^c_p(\{u_t\}_t,\{v_t\}_t),
\end{equation}
\begin{equation}\label{eq: 1-alpha_ineq}
\frac{d_p(v_s,\frac{s}{t}h_{t,\alpha})}{s} \leq \frac{d_p(v_t,h_{t,\alpha})}{t}=\frac{(1-\alpha) d_p(u_t,v_t)}{t} \leq (1-\alpha) d^c_p(\{u_t\}_t,\{v_t\}_t).
\end{equation}
We fix $\varepsilon >0$. Since $\frac{d_p(u_s,v_s)}{s} \nearrow d^c_p(\{u_t\}_t,\{v_t\}_t)$, \eqref{eq: alpha_ineq} and \eqref{eq: 1-alpha_ineq} imply existence of $s_{\alpha,\varepsilon}>0$ such that for any $s_{\alpha,\varepsilon} \leq s \leq t$ we have
$$\frac{d_p(u_s,\frac{s}{t}h_{t,\alpha})}{s} \leq (\alpha+\varepsilon)\frac{d_p(u_s,v_s)}{s}  \ \textup{ and } \ \frac{d_p(v_s,\frac{s}{t}h_{t,\alpha})}{s} \leq (1-\alpha+\varepsilon)\frac{d_p(u_s,v_s)}{s}.$$
Now Proposition \ref{prop: endpoint_distance_stability} implies that $d_p(h_{s,\alpha}, \frac{s}{t}h_{t,\alpha}) \leq \varepsilon ^{\frac{1}{r}} C  d_p(u_s,v_s)$ for any $s_{\alpha,\varepsilon} \leq s \leq t$. In particular, using \eqref{eq: ChCh3_metric_convex}, for any fixed $l>0$ such that $\max(l,s_{\alpha,\varepsilon}) \leq s \leq t$ we have
\begin{equation}\label{eq: main_est_chord}
\frac{d_p(\frac{l}{s}h_{s,\alpha}, \frac{l}{t}h_{t,\alpha})}{l} \leq \frac{d_p(h_{s,\alpha},\frac{s}{t} h_{t,\alpha})}{s} \leq \varepsilon ^{\frac{1}{r}} C\cdot \frac{d_p(u_s,v_s)}{s} \leq \varepsilon^{\frac{1}{r}}C d_p^c(\{u_t\}_t,\{v_t\}_t).
\end{equation}
By shrinking $\varepsilon$, the expression on the right can be chosen to be as small as we want, implying that the sequence $\{\frac{l}{t}h_{t,\alpha}\}_{t}\in \mathcal{E}^p_{\omega}$ is $d_p$-Cauchy. This is the crucial step! By \cite[Theorem 2]{Da15}, $(\mathcal{E}^p_{\omega},d_p)$ is complete, hence  $\lim_t \frac{l}{t}h_{t,\alpha} =: w_{l,\alpha} \in \mathcal E^p_{\omega}$, as proposed.

Moreover, letting $t \to \infty$ on the left hand side of \eqref{eq: alpha_ineq} and \eqref{eq: 1-alpha_ineq}, we obtain that
$$\frac{d_p(u_s,w_{s,\alpha})}{s} \leq \alpha d^c_p(\{u_t\}_t,\{v_t\}_t) \ \ \ \textup{ and } \  \ \ \frac{d_p(v_s,w_{s,\alpha})}{s} \leq (1-\alpha) d^c_p(\{u_t\}_t,\{v_t\}_t), \ \ s >0.$$
Letting $s \to \infty$, together with the triangle inequality this gives 
$$d_p^c(\{u_t\}_t,\{v_t\}_t)=d_p^c(\{u_t\}_t,\{w_{t,\alpha}\}_t) + d_p^c(\{w_{t,\alpha}\}_t,\{v_t\}_t),$$
ultimately implying that $d_p^c(\{u_t\}_t,\{w_{t,\alpha}\}_t) = \alpha d^c_p(\{u_t\}_t,\{v_t\}_t)$ and $ d_p^c(\{w_{t,\alpha}\}_t,\{v_t\}_t)=(1-\alpha) d^c_p(\{u_t\}_t,\{v_t\}_t)$. Suppose now that $0 \leq \alpha \leq \beta \leq 1$. These last two identities together with the triangle inequality give that
\begin{equation}\label{eq: geod_ineq}
(\beta - \alpha) d^c_p(\{u_t\}_t,\{v_t\}_t) \leq d_p^c(\{w_{t,\beta}\}_t,\{w_{t,\alpha}\}_t).
\end{equation}
To finish the proof we show that  equality holds in this estimate. Indeed, another application of \eqref{eq: ChCh3_metric_convex} gives that
$$\frac{d_p(\frac{l}{s}h_{s,\alpha},\frac{l}{s}h_{s,\beta})}{l} \leq \frac{d_p(h_{s,\alpha},h_{s,\beta})}{s}=\frac{(\beta - \alpha) d_p(u_s,v_s)}{s}, \ \ s >0.$$
Letting $s \to \infty$ in this estimate, and after that $l \to \infty$, the reverse inequality in \eqref{eq: geod_ineq} follows, finishing the proof. 
\end{proof}

The $d_p^c$-geodesic segment $[0,1] \ni \alpha \to \{w_{t,\alpha}\}_t \in \mathcal R^p_\o$ constructed in the above theorem will be called the $d_p^c$-\emph{chord} joining $\{w_{t,0}\}$ and $\{w_{t,1}\}$, as this curve is reminiscent of the chords joining the different points in the unit sphere of $\mathbb{R}^n$. 

Finally, using approximation, we point out that the same result holds for $p=1$ as well. First we remark that $d^c_p$-chords are automatically $d^c_{p'}$-chords for any $p' \leq p$. This observation is of independent interest, and is the ``radial version" of a well known phenomenon for the family of metric spaces $(\mathcal E^p_\o,d_p), \ p \geq 1$:

\begin{proposition}\label{prop: geod_sharing} Let $1 \leq p' < p$ and $\{u_t\}_t,\{v_t\}_t \in \mathcal R^p_{\omega}$. Trivially  $\{u_t\}_t,\{v_t\}_t \in \mathcal R^{p'}_{\omega}$, and the $d_p^c$-chord $[0,1] \ni \alpha \to \{w_{t,\alpha}\}_t \in \mathcal R^p_{\omega}$ connecting $\{u_t\}_t,\{v_t\}_t$ is also a $d_{p'}^c$-chord.
\end{proposition}

\begin{proof} 
To start, we trace the steps in the proof of Theorem \ref{thm: chord p} and notice that the curves $\alpha \to h_{t,\alpha}$, introduced in the argument, did not depend on the particular choice of $p$.

Fixing $l \geq 0$ and $\alpha \in [0,1]$, the crux of the proof is the fact  that  $d_p\big(\frac{l}{s}h_{s,\alpha},\frac{l}{s}h_{s,\alpha}\big) \to 0$ as $s,t \to \infty$, which follows from uniform convexity (in case $p>1$), as elaborated in \eqref{eq: main_est_chord}. Since $1 \leq p'< p$, we have that $d_{p'}(\cdot,\cdot) \leq d_p(\cdot,\cdot)$ and $\mathcal E^p_\o \subset \mathcal E^{p'}_\o$, hence  the same conclusion holds for $p'$ as well: $$d_{p'}\Big(\frac{l}{s}h_{s,\alpha},\frac{l}{t}h_{t,\alpha}\Big) \leq d_{p}\Big(\frac{l}{s}h_{s,\alpha},\frac{l}{t}h_{t,\alpha}\Big)\to 0 \ \textup{ as } \ s,t \to 0.$$
The rest of the proof does not use uniform convexity, and goes through without any difficulties for $p'$ in place of $p$, arriving at the conclusion that the chord $[0,1] \ni \alpha \to \{w_{t,\alpha}\}_t \in \mathcal R^p_{\omega} \subset \mathcal R^{p'}_{\omega}$ is a $d^{c}_{p'}$--chord as well.
\end{proof}

\begin{theorem}\label{thm: chord_L1} $(\mathcal R^1_{\omega}, d_1^c)$ is a complete geodesic metric space. Moreover, the $d_1^c$-chords of this space can be constructed by the method of Theorem \ref{thm: chord p}.
\end{theorem}

\begin{proof}
Given $\{w_{t,0}\}_t,\{w_{t,1}\}_t\in \mathcal{R}_{\omega}^1$, we will show that there exists a $d_1^c$-chord $[0,1]\ni \alpha\mapsto \{w_{t,\alpha}\}_t\in \mathcal{R}_{\omega}^1$ joining $\{w_{t,0}\}_t$ and $\{w_{t,1}\}_t$. 

Fix any $p>1$. Using Theorem \ref{thm: approximation of d1 rays} we can find $\{w^{k}_{t,0}\}_t,\{w^{k}_{t,1}\}_t \in \mathcal R^{p}_{\omega} \subset \mathcal R^{1}_{\omega}$  such that $w^{k}_{t,0} \searrow w_{t,0} $ and $w^{k}_{t,1} \searrow w^{1}_t $ for all $t \geq 0$. Let $[0,1] \ni \alpha \to \{w^{k}_{t,\alpha}\}_t \in \mathcal R^{p}_{\omega} \subset \mathcal R^{1}_{\omega}$ be the $d^c_{1}$-geodesic joining $\{w^{k}_{t,0}\}_t,\{w^{k}_{t,1}\}_t$, which exists by Proposition \ref{prop: geod_sharing}. 

We look at the construction of the curves $\alpha \to \{w^k_{t,\alpha}\}$  in the proof of Theorem \ref{thm: chord p} and attempt to construct $\alpha \to \{w_{t,\alpha}\}$ using the same method.

Using the fact that $d_1(u,v)=I(u)-I(v)$ for $u \geq v$, and affinity of $I$ along finite energy geodesics, one deduces that for any $\alpha \in [0,1]$ and $0 \leq s < t$ we have 
\begin{flalign}\label{eq: szaz}
d_1\Big(\frac{s}{t}h_{t,\alpha}^{k},\frac{s}{t}h_{t,\alpha}\Big) & = I\Big(\frac{s}{t}h_{t,\alpha}^{k}\Big) - I\Big(\frac{s}{t}h_{t,\alpha}\Big) \nonumber \\
& = \frac{s}{t}I\big(h_{t,\alpha}^{k}\big) - \frac{s}{t}I\big(h_{t,\alpha}\big)\\
&= \frac{s(1-\alpha)}{t}(I(w^k_{t,0}) - I(w_{t,0})) + \frac{s\alpha }{t} (I(w^k_{t,1}) - I(w_{t,1})) \nonumber\\
&= s(1-\alpha)(I(w^k_{1,0}) - I(w_{1,0})) + s \alpha (I(w^k_{1,1}) - I(w_{1,1})),\nonumber
\end{flalign}
with the last expression converging to zero regardless of the values of $t>0$. From here we get that $d_1\big(\frac{s}{t}h_{t,\alpha}^{k},\frac{s}{t}h_{t,\alpha}\big) \to 0$ as $k \to \infty$, uniformly with respect to $t$. 

On the other hand, by Proposition \ref{prop: geod_sharing} (and its proof) we get that $d_1^c(\frac{s}{t}h_{t,\alpha}^{k},w^k_{s,\alpha}) \to 0$ as $t \to \infty$ for any fixed $k \geq 0$.

By construction, each sequence $\{w^k_{s,\alpha}\}_k \in \mathcal E^1_\o$ is decreasing and $d_1$-bounded, hence by \cite[Lemma 4.16]{Da15} there exists $\{w_{s,\alpha}\} \in \mathcal E^1_\o$ such that $d_1(w^k_{s,\alpha},w_{s,\alpha}) \to 0$ as $k \to \infty$.

Lastly, the triangle inequality gives: 
$$d_1\Big(\frac{s}{t}h_{t,\alpha},w_{s,\alpha}\Big) \leq d_1\Big(\frac{s}{t}h_{t,\alpha}^{k},\frac{s}{t}h_{t,\alpha}\Big) + d_1^c\Big(\frac{s}{t}h_{t,\alpha}^{k},w^k_{s,\alpha}\Big) + d_1(w^k_{s,\alpha},w_{s,\alpha}).$$ 
Putting everything together, for $s \geq 0$ fixed, the first and last term on the  right hand side can be made arbitrarily small for big $k$. Next, with $k$ fixed, the same is true for the middle term for big $t$, i.e., $d_1\big(\frac{s}{t}h_{t,\alpha},w_{s,\alpha}\big)   \to 0$ as $t \to \infty$. 

As pointed out in the proof of Proposition \ref{prop: geod_sharing}, with this last fact in hand the rest of the proof of Theorem \ref{thm: chord p} goes through without any issues for $p=1$.
\end{proof}

\paragraph{Convexity of the radial K-energy. }Let $p \geq 1$. The \emph{radial K-energy} is defined for any $\{u_t\}_t \in \mathcal R^p_\o$,  and is given by the expression
$$\mathcal K\{u_t\} := \lim_{t\to \infty}\frac{\mathcal K(u_t)}{t},$$ 
where $\mathcal K: \mathcal E^p_\o \to (-\infty,\infty]$ is the extended K-energy of Mabuchi from \cite{BBEGZ11,BDL17}. In the setting of unit speed geodesics, this definition agrees with the $\yen$ invariant of \cite{ChCh3}. Also, there is clear parallel with the non-Archimedean K-energy of \cite{Bo18} (and references therein). 
\begin{lemma}
Let $\{u_t\}_t \in \mathcal R^p_u,\{v_t\}_t \in \mathcal R^p_v$ parallel, with $u,v\in \mathcal{E}^p$ satisfying $\mathcal{K}(u)<\infty$ and $\mathcal{K}(v)<+\infty$. Then $\mathcal{K}\{u_t\}=\mathcal{K}\{v_t\}$. 
\end{lemma}
\begin{proof}By the proof of Proposition \ref{prop: parallel} we can assume that either $u \leq v$ or $v \leq u$.

For each $t>0$ let $[0,t] \ni l \mapsto v^t_l \in \mathcal{E}^p_{\omega}$ be the finite energy geodesic connecting $v^t_0:=v$ and $v^t_t:=u_t$.  It follows from Proposition \ref{prop: parallel} (and its proof) that $\lim_{t\to +\infty} d_p(v^t_l,v_l) =0$ for each $l$ fixed. 
 By convexity of $\mathcal{K}$ \cite[Theorem 1.2]{BDL17}, for any $0<l<t$ we have that
$$
\mathcal{K}(v^t_l) \leq \Big(1-\frac{l}{t}\Big) \mathcal{K}(v)+ \frac{l}{t} \mathcal{K}(u_t). 
$$
Thus, letting $t\to +\infty$ and using lower semicontinuity of $\mathcal{K}$ w.r.t. $d_p$ we obtain 
$$
\frac{\mathcal{K}(v_l)}{l} \leq \frac{\mathcal{K}(v)}{l} + \mathcal{K}\{u_t\}. 
$$
Letting $l \to +\infty$ yields $\mathcal{K}\{v_t\}\leq \mathcal{K}\{u_t\}$. The reverse inequality is obtained by reversing the roles of $u,v$. 
\end{proof}

By the above lemma it makes sense to restrict to $\mathcal R^p_{\omega}$ when considering the radial K-energy. Since $d^c_p$-convergence implies $d^c_1$-convergence it follows from \cite[Proposition 5.9]{ChCh3} that the resulting functional $$\mathcal K\{\cdot \}:\mathcal R^p_{\omega} \to (-\infty,\infty]$$
 is $d_p^c$-lsc. In the last result of this section we point out that $\mathcal K\{\cdot\}$ is also convex along the chords of $\mathcal R^p_{\omega}$ for any $p\geq 1$:
\begin{theorem}\label{thm: radialK_convexity} Suppose that $p\geq 1$ and $[0,1] \ni \alpha \to \{w_{t,\alpha}\}_t \in \mathcal R^p_{\omega}$ is a $d_p$-chord joining $\{u_t\}_t,\{v_t\}_t \in \mathcal R^p_{\omega}$. Then $\alpha \to \mathcal K\{w_{t,\alpha}\}$ is convex.
\end{theorem}
\begin{proof} We use the notation and terminology of the proof of Theorems \ref{thm: chord p} and \ref{thm: chord_L1}, and normalize $\mathcal K$ such that $\mathcal K(0)=0$. Using convexity of $\mathcal K$ along finite energy geodesics \cite[Theorem 1.2]{BDL17} we know that for any $0 < s \leq t$ and $\alpha \in [0,1]$ we have
$$\frac{\mathcal K\big(\frac{s}{t}h_{t,\alpha}\big)}{s} \leq \frac{\mathcal K(h_{t,\alpha})}{t} \leq (1-\alpha) \frac{\mathcal K(u_t)}{t}  + \alpha \frac{\mathcal K(v_t)}{t}.$$
Since $d_p(\frac{s}{t}h_{t,\alpha},w_{s,\alpha}) \to 0$, given that $\mathcal K$ is $d_p$-lsc (\cite[Theorem 1.2]{BDL17}) it follows that
$$\frac{\mathcal K(w_{s,\alpha})}{s} \leq \liminf_{t \to \infty} \frac{\mathcal K\big(\frac{s}{t}h_{t,\alpha}\big)}{s} \leq (1-\alpha) \mathcal K \{u_t \}  + \alpha \mathcal K\{v_t\}.$$
The result now follows after taking the limit  $s \to \infty$.
\end{proof}

\begin{remark} Many theorems that hold for the finite energy metric spaces $(\mathcal E^p_\o,d_p)$ admit a radial version for $(\mathcal R^p_\o,d_p)$. As we already pointed out, Theorem \ref{thm: R_p_complete_geod_K_convex_intr}, Lemma \ref{lem: decreasing_ray_d_plimit}, and also Theorem \ref{thm: K_approx} below are examples of this phenomenon. This does not seem to be limited to only these results either. Indeed, though we will not pursue this further here, one can introduce radial analogs of the operators $\max(\cdot,\cdot)$ and $P(\cdot,\cdot)$, and similar identities/inequalities/results hold for these as the ones described in \cite{Da15,Da17}.
\end{remark}

\section{Approximation with converging radial K-energy}

\subsection*{Approximation with rays of bounded potentials}

The goal of this subsection is to strengthen the conclusion of Theorem \ref{thm: approximation of d1 rays} and obtain Theorem \ref{thm: approximation_dp_intr}(i) in the process:
\begin{theorem}
 	\label{thm: K_approx} Let $\{ u_t\}_t \in \mathcal R^p_{\omega}$, $p \geq 1$. Then there exists a sequence $\{ u^j_t\}_t \in \mathcal R_\o^\infty$ such that  $u_t^j$ decreases to $u_t$, for each $t>0$ fixed and $ \mathcal K\{u^j_t\} \to \mathcal K\{u_t\}$. In particular $\lim_{j\to +\infty}d_p^c(\{u_t^j\},\{u_t\})=0$. 
\end{theorem}

In case $\mathcal K\{u_t\}=+\infty$, by the fact that $\mathcal K\{\cdot\}$ is $d_p^c$-lsc \cite[Proposition 5.9]{ChCh3}, we will simply invoke Theorem \ref{thm: approximation of d1 rays} for the existence of the sequence of $\{u^j_t\}_t$. If  $\mathcal K\{u_t\}$ is finite, we will need a much more delicate argument, resting on the relative Ko{\l}odziej type estimate of \cite[Theorem 3.3]{DDL4}, as detailed in the argument below. 

At places, the argument below shares some similarities with the proof of \cite[Theorem 3.2]{DH17}, with the the relative Ko{\l}odziej type estimate of \cite{DDL4} taking the place of Perelman's estimates along the K\"ahler--Ricci flow on Fano manifolds. Before engaging in the proof of Theorem \ref{thm: K_approx}, we prove an auxiliary lemma:

\begin{lemma}\label{lem: truncated_limit} Let $\{u_t\}_t \in \mathcal R^1_\o$ with $\sup_X u_t =0, \ t \geq 0$. Then 
\begin{equation} \label{eq: truncated_limit}
	\lim_{j \to \infty} \limsup_{t\to +\infty} \int_{\{u_t \leq -jt\}} \frac{(-u_t)}{t} \omega_{u_t}^n =0. 
\end{equation}
\end{lemma}
\begin{proof}	Indeed, it follows from Theorem \ref{thm: approximation of d1 rays} that we can choose $\{u^j_t\}_t \in \mathcal R^1_{\omega}$ such that $u_t \leq \max(u_t, -jt) \leq u^j_t \leq 0$ and $d^c_1(\{u^j_t\},\{u_t\}) = \lim_{t \to \infty} \frac{I(u^j_t) - I(u_t)}{t} = 0$. From monotonicity and elementary properties of $I(
\cdot)$ we conclude that  $\lim_{t \to \infty} \frac{I(\max(u_t, -jt)) - I(u_t)}{t} = 0$, ultimately implying
$$0 \leq \lim_j \lim_{t \to \infty }\int_X \frac{\max (u_t, - jt) - u_t}{t} 
\omega_{u_t}^n \leq  (n+1)\lim_j \lim_{t \to \infty} \frac{I(\max(u_t, -jt)) - I(u_t)}{t}=0.$$
Consequently both limits are equal to zero, and  on the set $\{u_t \leq -2jt\}$,  we have  $0\geq \max(u_t,-jt)-u_t \geq -\frac{u_t}{2}$. This and the above together yield \eqref{eq: truncated_limit}.
\end{proof}
\begin{proof}[Proof of Theorem \ref{thm: K_approx}]
Using Theorem \ref{thm: approximation of d1 rays} and the fact that $\mathcal K\{\cdot\}$ is $d_p^c$-lsc, we can assume that $\mathcal{K}\{u_t\} <+\infty$. Also, via Lemma \ref{lem: sup_X_linear}, by possibly adding $Ct$ to $u_t$ we can additionally assume that $\sup_X u_t=0$, i.e., $t \to u_t$ is $t$-decreasing with $u_\infty := \lim_{t \to 
\infty} u_t \in \textup{PSH}(X,\omega)$.

 For each $j>1,l>1$, we let $\varphi_l^j \in \mathcal E(X,\omega)$ be the unique $\omega$-psh function, whose existence is guaranteed by \cite[Theorem A]{GZ07}, such that
\begin{equation}\label{eq: varphi_l_def}
	\omega_{\varphi_l^j}^n = \bigg( 1 - \frac{1}{2^j}\bigg){\mathbbm 1}_{\{u_l>-jl\}} \omega_{u_l}^n + a_{j,l} \omega^n, \ \sup_X \varphi_l^j =0, 
\end{equation}
where $0\leq a_{j,l} \leq 1$ is a constant arranged so that the measure on the right hand side has total mass equal to $\int_X \omega^n$.  

Next we point out that the conditions of Theorem \ref{thm: uniform_estimate_Kolodziej}  are satisfied with appropriate choice of data. Let $a:= \big(1 -\frac{1}{2^j} \big)^{{1}/{2}} \in (0,1), \  u:= \varphi^j_l,\ \chi := \big(1 -\frac{1}{2^j} \big)^{{1}/{2n}} \max(u_l,-jl)$,  and $f:=1$. Then, using locality of the non-pluripolar complex Monge--Amp\`ere measure (see e.g. \cite[Corollary 1.7]{GZ07}) we have that
$
\omega_{\max(u_l,-jl)}^n \geq {\mathbbm 1}_{\{u_l>-jl\}} \omega_{u_l}^n,
$
hence,
$$\omega^n_{u} \leq a \omega_\chi^n + f \omega^n.$$
Moreover, due to \cite[Proposition 4.3]{BEGZ10} and  \cite[Lemma 4.2]{DDL4}, there exists $A(X,\omega) >0$ such that for any Borel set $E\subset X$  we have  
$$\int_E f \omega^n = \int_E \omega^n \leq A \textup{Cap}_\o(E)^2 \leq A \bigg( 1 -\bigg(1 -\frac{1}{2^j} \bigg)^{{1}/{2n}}\bigg)^{-2n} \textup{Cap}_\chi(E)^2,$$
where $\textup{Cap}_\omega$ is the usual Monge--Amp\`ere capacity  and $\textup{Cap}_\chi$ is its relative version from \cite[Section 3]{DDL4}. Lastly, we note that $\chi \leq 0 =P[\varphi^j_l]$, due to \cite[Theorem 3]{Da13}, hence all the conditions of Theorem \ref{thm: uniform_estimate_Kolodziej}  are satisfied to imply that 
	\begin{equation}
	\label{eq: DDL4 estimate}
    \varphi_{l}^j=u  \geq  \chi - C_j \geq \max(u_l,-jl) - C_j,  
	\end{equation}where $C_j>0$ is a constant depending on $j$, but not $l > 1$! In particular $\varphi^j_l$ is bounded. 
		
Moreover, for $1<j<k$ and $l>1$ we have 
		$$
		\omega_{\varphi^j_l}^n \leq \frac{1-2^{-j}}{1-2^{-k}} \omega_{\varphi_l^k}^n + \omega^n.
		$$
Similarly to \eqref{eq: varphi_l_def}, this allows for another application of Theorem \ref{thm: uniform_estimate_Kolodziej}, with the choice of data $a:= \big(\frac{1 - 2^{-j}}{1 - 2^{-k}} \big)^{{1}/{2}} \in (0,1), \ u:= \varphi^j_l, \ \chi := \big(\frac{1 - 2^{-j}}{1 - 2^{-k}} \big)^{{1}/{2n}} \varphi^k_l$,  and $f:=1$. Similarly to the above, the conditions of Theorem \ref{thm: uniform_estimate_Kolodziej} are satisfied to yield that
\begin{equation}
	\label{eq: DDL4 estimate bis}
    \varphi_{l}^j = u  \geq  \chi - C_{j,k} \geq \varphi_l^k - C_{j,k},  
	\end{equation} where $C_{j,k}>0$ depends on $j,k$, but not on $l>1$!
	
	For each $l>1$ let $[0,l] \ni t \mapsto u^{j,l}_t$ be the bounded geodesic segment joining $0$ and $\varphi_l^j +C_j$.   Then \eqref{eq: DDL4 estimate} and \eqref{eq: DDL4 estimate bis} together with the comparison principle for finite energy geodesics implies that 
\begin{equation}\label{eq: C_0_geod_est}
 \frac{C_j t}{l} \geq u^{j,l}_t \geq \max(u_t, -jt),  \ \ \ t\in [0,l],
\end{equation} 
and 
\begin{equation}\label{eq: C_0_geod_est decreasing}
 \frac{C_{j} t}{l} \geq u^{j,l}_t \geq u^{k,l}_t - \frac{D_{j,k}t}{l} ,  \ 0<j<k, \ \ t\in [0,l],
\end{equation} 
where $D_{j,k}$ depends on $j,k$ but not on $l>1$.

To show that the above geodesic sequences subconverge to appropriate geodesic rays, we first prove a number of estimates in the claims below. \vspace{0.1cm}\\
\noindent{\bf Claim 1.}  For any $j >1$ we have 
$$
\limsup_{t\to +\infty}  \frac{ {\rm Ent}\big(\omega^n,\omega_{\varphi_t^j}^n\big)}{t} \leq  \limsup_{t\to +\infty} \frac{{\rm Ent}(\omega^n,\omega_{u_t}^n)}{t}.
$$
Since ${\rm Ent}(\omega^n,\omega_{u_t}^n)<+\infty$, for any $t \geq 0$, we can write $\omega_{\varphi_t^j}^n = f_{t,j} \omega^n$ and $\omega_{u_t}^n = f_t \omega^n$. Observe first that for any $g_t \geq 0$ with $\int_X g_t \omega^n =\int_X \omega^n$ we have that 
$$
\limsup_{t\to +\infty} \int_X \frac{g_t \log (g_t)}{t} \omega^n = \limsup_{t\to +\infty} \int_X \frac{(g_t+B) \log (g_t+B)}{t} \omega^n, \ \forall B\geq 1.  
$$
This follows after splitting up both integrals using the partition $\{0 \leq g_t \leq C\}$ and $\{C < g_t\}$ for $C>0$ big and noticing that the $\limsup$ of integrals on $\{0 \leq g_t \leq C\}$ is always zero. 

By construction, $1\leq f_{t,j}+1 \leq f_t + 2$ and hence, since $s\mapsto s \log (s), s >1$ is increasing,  $(f_{t,j}+1) \log (f_{t,j}+1) \leq (f_t+2) \log (f_t + 2)$.  Using the above we then conclude:
\begin{flalign*}
\limsup_{t\to +\infty} \int_X \frac{f_{t,j}\log (f_{t,j})}{t} \omega^n &=\limsup_{t\to +\infty} \int_X \frac{(f_{t,j}+1)\log (f_{t,j}+1)}{t} \omega^n\\
& \leq \limsup_{t\to +\infty} \int_X \frac{(f_t+2) \log (f_t + 2)}{t} \omega^n \\
&  = \limsup_{t\to +\infty} \int_X \frac{f_t \log f_t }{t}\omega^n.
\end{flalign*}
\noindent
{\bf Claim 2.}  We have 
$$
\lim_j \limsup_{t\to +\infty}  \frac{ \mathcal{I}(\varphi_t^j,u_t)}{t} = 0. 
$$
Before we start with the argument, we recall that $\mathcal I(v,w)=\int_X (v-w)(\o_w^n - \omega_v^n)$ for $v,w \in \mathcal E^1_\o$. By \eqref{eq: varphi_l_def} we have  
$$ \mathcal{I}(\varphi_t^j,u_t) \leq  \frac{1}{2^j}\int_X |\varphi^t_j - u_t| \omega_{u_t}^n + \int_{\{u_t \leq -jt\}} |\varphi^t_j - u_t| \omega_{u_t}^n + \int_X |\varphi^t_j - u_t| \omega^n,$$ and the claim follows from the following three estimates. First, the estimate of \eqref{eq: DDL4 estimate} and basic properties of $I(\cdot)$ give that
\begin{equation}\label{eq: AM_est_int}
\lim_j \limsup_{t\to +\infty} \frac{1}{2^j} \int_X \frac{|\varphi^t_j - u_t|}{t} \omega_{u_t}^n \leq \lim_j \limsup_{t\to +\infty}\frac{1}{2^j} \frac{C_j}{t} +\lim_j \limsup_{t\to +\infty} \frac{1}{2^j}\frac{|I(u_t)|}{t}=0.
\end{equation}
Second, by the dominated convergence theorem we have that
\begin{equation} \label{eq: L_1_est_int}
\lim_{t\to +\infty}\int_X \frac{ |\varphi_t^j -u_t|}{t} \omega^n  \leq \lim_{t\to +\infty}\int_X \frac{ |C_j -u_\infty|}{t} \omega^n =0.
\end{equation}
Third, by Lemma \ref{lem: truncated_limit}  and \eqref{eq: DDL4 estimate},
\begin{equation}\label{eq: trunc_est}
\lim_j \limsup_{t\to +\infty} \int_{\{u_t \leq -jt\}} \frac{|\varphi_t^j -u_t|}{t} \omega_{u_t}^n   \leq  \lim_j \limsup_{t\to +\infty} \int_{\{u_t \leq -jt\}}  \frac{|u_t| + C_j}{t} \omega_{u_t}^n =0.
\end{equation}
\noindent {\bf Claim 3.}  We have 
	$$
	\lim_j \limsup_{t\to +\infty}  \frac{ | I(\varphi_t^j)  -I(u_t) |}{t} =0.$$

This claim follows from Claim 2 and Lemma \ref{lem: BBGZ I energy} below, with $\varphi_1=\varphi_t^j$, $\varphi_2 = u_t$ and $\psi=0$.  Indeed, given \eqref{eq: DDL4 estimate}, we have that  $\max (-I(\varphi_t^j),-I(u_t)) \simeq Ct +C_j$, for a uniform constant $C$. Lemma \ref{lem: BBGZ I energy} then gives 
$$
\left |\int_X (\varphi_t^j-u_t)(\omega_{u_t}^n -\omega^n) \right | \leq (Ct +C_j ) f\left (\mathcal{I}(\varphi_t^j,u_t)/(Ct+C_j)\right). 
$$  
Hence 
$$
\lim_j \limsup_{t\to +\infty}  \bigg|  \int_X \frac{(\varphi_t^j-u_t)}{t} (\omega_{u_t}^n-\omega^n) \bigg| =0.  
$$
Again, due to \eqref{eq: DDL4 estimate} and elementary properties of $I(\cdot)$ we have that 
$$
0\leq \limsup_{t\to +\infty} \frac{I(\varphi_t^j)- I(u_t)}{t} \leq \limsup_{t\to +\infty} \int_X \frac{\varphi_t^j-u_t}{t} \omega_{u_t}^n. 
$$
Putting these last two estimates together and \eqref{eq: L_1_est_int}, the claim follows. \vspace{0.1cm}\\
\noindent
	{\bf Claim 4.}  For any  closed smooth real $(1,1)$-form $\alpha$  we have 
$$
\lim_j 	 \limsup_{t\to +\infty}  \frac{ |I_{\alpha}(\varphi_t^j) - I_{\alpha} (u_t)|}{t}  =0.
$$
Recall that $
I_{\alpha}(v) := \sum_{j=0}^{n-1} \int_X v \alpha \wedge \omega_{v}^j.$
Since $\alpha$ can be written as the difference of two K\"ahler forms, and $I_\alpha(\cdot)$ is monotone when $\alpha \geq 0$, notice that the claim follows if we can argue that
$$
\lim_j \limsup_{t\to +\infty}  \frac{|I_{\omega}(\varphi_t^j+C_j) -I_{\omega}(u_t)|}{t}=\lim_j \limsup_{t\to +\infty}  \frac{I_{\omega}(\varphi_t^j+C_j) -I_{\omega}(u_t)}{t} =0. 
$$
Using \eqref{eq: DDL4 estimate} we observe that this last identity is a consequence of 
$$
\lim_j \limsup_{t\to +\infty} \frac{I_{\omega}(\varphi_t^j+C_j)-I_{\omega}(u_t)}{t} =\lim_j \limsup_{t\to +\infty} \frac{I_{\omega}(\varphi_t^j)-I_{\omega}(u_t)}{t}=0. 
$$
However we have that
$$\frac{I_\omega(\varphi^j_t) - I_\omega(u_t)}{t} = \frac{(n+1)(I(\varphi^j_t) - I(u_t))}{t} - \frac{\int_X \varphi^j_t \omega_{\varphi_{t}^j}^n - \int_X u_t \omega_{u_{t}}^n}{t}.$$
So, by Claim 3, it is enough to show that 
$$
\lim_j \limsup_{t\to +\infty}  \frac{\big| \int_X \varphi_t^j \omega_{\varphi_t^j}^n -\int_X u_t \omega_{u_t}^n  \big| }{t} =0. 
$$
Again, due to \eqref{eq: DDL4 estimate} and elementary properties of $I(\cdot)$ we have that 
$$
\lim_j \limsup_{t \to +\infty}  \frac{\big| \int_X (\varphi_t^j -u_t) \omega_{\varphi_t^j}^n \big| }{t} \leq \lim_j \limsup_{t\to +\infty} \frac{ I(\varphi_t^j)  -I(u_t) }{t} =0,
$$
where the last identity follows from Claim 3. Due to \eqref{eq: varphi_l_def} and \eqref{eq: L_1_est_int} we also have 
$$
\lim_j  \limsup_{t \to +\infty}  \frac{\left  | \int_X u_t  \left(\omega_{\varphi_t^j}^n- \omega_{u_t}^n\right) \right | }{t} \leq \lim_j \limsup_{t\to +\infty} \bigg( \frac{1}{2^j} \int_X \frac{|u_t|}{t} \omega_{u_t}^n + \int_{\{u_t\leq -jt\}} \frac{|u_t|}{t} \omega_{u_t}^n\bigg)=0,
$$
where the last equality follows from Lemma \ref{lem: truncated_limit} and the fact that $\int_X |u_t| \omega_{u_t}^n \leq d_1(0,u_t)=td_1(0,u_1)$ (\cite[Theorem 3]{Da15}). \vspace{0.1cm}

\noindent {\bf Conclusion.} To start, we recall the Chen--Tian formula for the K-energy that extends to $\mathcal E^1_\o$ (see \cite[Theorem 1.2]{BDL17}):
$$\mathcal K(u) =\textup{Ent}(\omega^n,\omega_u^n) + \overline{S} I(u) -n I_{\textup{Ric }\o}(u), \ \ u \in \mathcal E^1_\o.$$
There exists an increasing sequence $l_k \to +\infty$ such that 
$
\lim_{k\to +\infty} {\Ent\big(\omega^n, \omega_{u_{l_k}}^n\big)}/{l_k} = \limsup_{t\to +\infty}  {\Ent\big(\omega^n, \omega_{u_{t}}^n\big)}/{t}.
$
It then follows that 
\begin{flalign*}
\limsup_{k\to +\infty}  & \frac{\Ent\Big(\omega^n, \omega_{\varphi_{l_k}^j}^n\Big)-  \Ent\Big(\omega^n, \omega_{u_{l_k}}^n\Big)}{l_k}    \leq  \limsup_{k\to +\infty}  \frac{\Ent\Big(\omega^n, \omega_{\varphi_{l_k}^j}^n\Big)}{l_k} - \lim_{k} \frac{\Ent\left (\omega^n, \omega_{u_{l_k}}^n\right )}{l_k}\\
& =  \limsup_{k\to +\infty}  \frac{\Ent\Big(\omega^n, \omega_{\varphi_{l_k}^j}^n\Big)}{l_k} - \limsup_{t\to +\infty} \frac{\Ent\left (\omega^n, \omega_{u_{t}}^n\right )}{t} \\
& \leq  \limsup_{t\to +\infty}  \frac{\Ent\left (\omega^n, \omega_{\varphi_{t}^j}^n\right )}{t} - \limsup_{t\to +\infty} \frac{\Ent\left (\omega^n, \omega_{u_{t}}^n\right )}{t}.
\end{flalign*}
It thus follows from Claim 1 that 
\begin{equation}
    \label{eq: Step 1 CVG}
    \lim_{j\to +\infty} \limsup_{k\to +\infty}  \frac{\Ent\Big(\omega^n, \omega_{\varphi_{l_k}^j}^n\Big)-  \Ent\left (\omega^n, \omega_{u_{l_k}}^n\right )}{l_k} \leq 0. 
\end{equation}
We continue with
$$\limsup_{k\to +\infty} \frac{\mathcal{K}(\varphi_{l_k}^j)}{l_k}  \leq \limsup_{k\to +\infty} \frac{\mathcal{K}(\varphi_{l_k}^j)-\mathcal{K}(u_{l_k})}{l_k} + \limsup_{k\to +\infty} \frac{\mathcal{K}(u_{l_k})}{l_k}.$$
Thus, using the Chen--Tian formula  together with \eqref{eq: Step 1 CVG} and the estimates of Claims 3, 4, we can continue to write that
$$
\lim_{j\to +\infty} \limsup_{k \to +\infty} \frac{\mathcal{K}(\varphi_{l_k}^j)}{l_k} \leq \mathcal{K}\{u_t\}. 
$$
As a result, there exists an increasing sequence $\{j_m\}_m \subset \Bbb N$ such that 
$$
\limsup_{k \to +\infty} \frac{\mathcal{K}(\varphi_{l_k}^{j_m})}{l_k} \leq \mathcal{K}\{u_t\}  + \frac{1}{m}. 
$$
Hence, returning to the geodesic segments constructed at the beginning of the argument, by convexity of the K-energy we have, for all $t\in [0,l_k]$, 
\begin{equation}\label{eq: K-en_ineq}
\limsup_{k\to +\infty}\frac{\mathcal{K} (u^{j_m,{l_k}}_t) }{t} \leq \limsup_{k\to +\infty}\frac{\mathcal{K}(\varphi_{l_k}^{j_m})}{l_k} \leq \mathcal{K}\{u_t\} +\frac{1}{m}.
\end{equation}
Let us fix $m \geq 1$  and $t \in \Bbb Q^+$ momentarily. We use the compactness property of $\mathcal E^1_\o$ (see \cite[Corollary 4.8]{BDL17}) to extract a subsequence (again denoted by $l_k=l_k(m,t)$) such that $d_1(u^{j_m,l_k}_t,u^m_t) \to 0$ as $k\to \infty$ for some $u^m_t \in \mathcal E^1_\o$. Using a diagonal Cantor process it is actually possible to pick the same subsequence of $\{l_k\}_k$ for each $m \geq 1$ and $t \in \Bbb Q^+$. Moreover, due to the endpoint stability of geodesic segments  \cite[Proposition 4.3]{BDL17}, we get that the convergence extends for all $t \geq 0$: there exists $u^m_t \in \mathcal E^1_\o$ such that $d_1(u^{j_m,l_k}_t,u^m_t) \to $ as $k \to \infty$ for any $t \geq 0$ and $\{u_t^m\}_t \in \mathcal R^\infty_\o$.

Now we prove additional properties for our sequence $\{u_t^m\}_t$. By \eqref{eq: C_0_geod_est},  we notice that $\{u_t^m\}_t \in \mathcal R^\infty_\o$:
\begin{equation}\label{eq: u_h_est_C_0}
\max(u_t,-j_mt) \leq u^m_t \leq 0.
\end{equation}
Moreover, by \eqref{eq: C_0_geod_est decreasing} we also have that $\{u^m_t\}_t$ is $m$-decreasing! 

Fixing $t>0$, since $d_1(u^{j_m,l_k}_t,u^m_t) \to 0$ as $k \to \infty$, due to $d_1$-lower semicontinuity of $\mathcal{K}$, from \eqref{eq: K-en_ineq} we obtain that 
\begin{equation}
\label{eq: approximation radial  K energy}
\frac{\mathcal{K} (u^{m}_t) }{t} \leq \mathcal{K}\{u_t\} + \frac{1}{m},  \ \ \forall t>0,
\end{equation}
hence $\mathcal{K}\{u_t^m\} \leq \mathcal{K}\{u_t\}+\frac{1}{m}$, as desired. 

Next, we argue that $d_1(u^m_t,u_t) \to 0$ for any $t \geq 0$, as $m \to \infty$. But this is simply a consequence of  Claim 3. Indeed, due to \eqref{eq: u_h_est_C_0}, we only need to argue that:
\begin{equation}\label{eq_last_I_limit}
\lim_{m\to +\infty} \frac{d_1(u^m_t,u_t)}{t}=\lim_{m\to +\infty}  \frac{I(u^m_t) - I(u_t)}{t} =0.
\end{equation}
But from $I$-linearity, for any $t \in [0,l_k]$ we have that
$$\frac{I(u^m_t) - I(u_t)}{t}=\limsup_{k \to \infty} \frac{I(u^{j,l_m}_t) - I(u_t)}{t} = \limsup_{k \to \infty} \frac{I(\varphi^{j_m}_{l_k}) - I(u_{l_k})}{l_k},$$
and the right hand side converges to zero  as $m \to +\infty$, by Claim 3.

Finally, $d_1(u^m_t,u_t) \to 0$ implies that $u^m_t \searrow u_t$ (\cite[Theorem 5]{Da15}), hence we can invoke Lemma \ref{lem: decreasing_ray_d_plimit} to conclude that  $d_p^c(\{u^m_t\}_t,\{u_t\}_t) \to 0$, as $m \to \infty$.
\end{proof}

In the above argument we have used the following lemma whose proof goes along the same lines as \cite[Theorem 5.8]{BBGZ13}: 
\begin{lemma}
	\label{lem: BBGZ I energy}
	 There exists a  continuous non-decreasing function $f: \mathbb{R}^+ \rightarrow \mathbb{R}^+$ with $f(0)=0$   such that for all $0\geq \varphi_1,\varphi_2,\psi\in \mathcal{E}^1_{\omega}$, we have 
	$$
	\left |\int_X (\varphi_1-\varphi_2) (\omega_{\varphi_2}^n -\omega_{\psi}^n) \right | \leq A f (\mathcal{I}(\varphi_1,\varphi_2)/ A),
	$$
	where $A= \max(-I(\varphi_1), -I(\varphi_2), -I(\psi),1)$.
\end{lemma}
In the proof below we use $C_n>0$ to denote various numerical constants (only dependent on $\dim X = n$) and $f : \mathbb{R}^+ \rightarrow \mathbb{R}^+$ to denote a continuous non-decreasing function such that $f(0)=0$. They may be different from place to place. 
\begin{proof}
By approximation of finite energy potentials from above by smooth ones, we can assume that $\varphi_1,\varphi_2,\psi$ are smooth (the convergence of the integrals is assured by the results of \cite{GZ07}, see for example \cite[Proposition 2.11]{Da18}). 
We set $u= \varphi_1-\varphi_2$ and $v= (\varphi_1+\varphi_2)/2$.  For $p=0,...,n$ let 
	$$
	a_p := \int_X u \omega_{\varphi_2}^p \wedge \omega_{\psi}^{n-p}  \ \ \ \ \textup{ and } \ \ \ 
	b_p := \int_X i \partial u \wedge \bar{\partial} u \wedge \omega_v^{p} \wedge \omega_{\psi}^{n-p-1}. 
	$$
	It follows from \cite[Proposition 2.5]{GZ07} that  
	\begin{equation}
		\label{eq: fundamental inequality GZ07}
		\mathcal{I}(\psi_1,\psi_2) \leq C_n  (|I(\psi_1)|+ |I(\psi_2)|), \ \ \textup{ for all } 0\geq \psi_1,\psi_2 \in \mathcal{E}^1_{\omega}. 
	\end{equation}
	In particular, $\mathcal{I}(\psi,\varphi_j) \leq C_n A$, $j=1,2$.  
	
	For $p=0,1,...,n-1$ we have, using integration by parts and by the Cauchy--Schwarz inequality,
	\begin{flalign*}
	|a_p &- a_{p+1}| =  \left |\int_X i \partial u \wedge  \bar{\partial}(\psi-\varphi_2) \wedge \omega_{\varphi_2}^p \wedge \omega_{\psi}^{n-p-1}\right | \\
	& \leq  \left ( \int_X i \partial u \wedge  \bar{\partial} u \wedge \omega_{\varphi_2}^p \wedge \omega_{\psi}^{n-p-1} \right )^\frac{1}{2} \left ( \int_X i \partial (\psi-\varphi_2) \wedge \bar{\partial}(\psi-\varphi_2) \wedge \omega_{\varphi_2}^p \wedge \omega_{\psi}^{n-p-1} \right )^\frac{1}{2}\\
	& \leq  C_n b_p^\frac{1}{2} \left( \mathcal{I}(\varphi_2,\psi)\right )^{1/2} \leq C_n b_p^{1/2} A^\frac{1}{2}.   
	\end{flalign*}
In the last line above we have used $\omega_{\varphi_2} \leq 2 \omega_v$ and  the inequality 
\begin{flalign*}
\int_X i \partial (\psi-\varphi_2)  & \wedge \bar{\partial} (\psi-\varphi_2)\wedge  \omega_{\varphi_2}^p \wedge \omega_{\psi}^{n-p-1} \leq  \int_X (\psi-\varphi_2) (\omega_{\varphi_2}^n-\omega_{\psi}^n). 
\end{flalign*}
It thus follows, by summing up the estimates of $|a_p-a_{p+1}|$ above for $p=0,...,n-1$, that 
\begin{equation}\label{eq: a_ineq}
|a_0-a_{n}| \leq C_n A^\frac{1}{2} \sum_{p=0}^{n-1} b_p^\frac{1}{2}. 
\end{equation}
We claim that there is a non-decreasing continuous function $f: \mathbb{R}^+ \rightarrow \mathbb{R}^+$ with $f(0)=0$ such that 
$$
b_p \leq A f(\mathcal{I}(\varphi_1,\varphi_2)/A), \ 0 \leq p \leq n-1.
$$
We proceed by (backwards) induction. For $p=n-1$ we can simply take $f(s) = C_n s$, $s\geq 0$. 
By the same argument as above using integration by parts and the Cauchy-Schwarz inequality we  have, for $0\leq p\leq n-2$, 	
\begin{flalign*}
	& b_p  - b_{p+1} =  \int_X i \partial u \wedge \bar{\partial} u \wedge i \partial \bar{\partial} (\psi-v) \wedge \omega_{v}^p \wedge \omega_{\psi}^{n-p-2}\\
	& \leq  \left |\int_X i\partial u \wedge \bar{\partial} (\psi-v) \wedge i \partial \bar{\partial}  u \wedge \omega_{v}^p \wedge \omega_{\psi}^{n-p-2}   \right |\\
	& \leq  \left |\int_X i\partial u \wedge \bar{\partial} (\psi-v) \wedge (\omega_{\varphi_1}-\omega_{\varphi_2}) \wedge \omega_{v}^p \wedge \omega_{\psi}^{n-p-2}   \right | \\ 
	& \leq  \left |\int_X i \partial u \wedge \bar{\partial} (\psi-v) \wedge \omega_{\varphi_1} \wedge \omega_{v}^p \wedge \omega_{\psi}^{n-p-2} \right | +  \left |\int_X i \partial u  \wedge \bar{\partial} (\psi-v) \wedge \omega_{\varphi_2} \wedge \omega_{v}^p \wedge \omega_{\psi}^{n-p-2} \right |    \\
	& \leq C_n \left (\int_X i \partial  u \wedge \bar{\partial} u  \wedge \omega_v^{p+1} \wedge \omega_{\psi}^{n-p-2} \right )^\frac{1}{2}  \left (\int_X i \partial (\psi-v)  \wedge \bar{\partial} (\psi-v) \wedge \omega_v^{p+1} \wedge \omega_{\psi}^{n-p-2} \right)^\frac{1}{2} \\
	& \leq  C_n \mathcal{I}(\psi,v)^\frac{1}{2} b_{p+1}^\frac{1}{2}, 
\end{flalign*}
where we used several times that $\omega_{\varphi_j} \leq 2 \omega_{v}$. Using\eqref{eq: fundamental inequality GZ07} we thus have 
$$
b_p \leq b_{p+1} + A C_n (b_{p+1}/A)^\frac{1}{2} \leq  A f (\mathcal{I}(\varphi_1,\varphi_2)/A) + A C_n f(\mathcal{I}(\varphi_1,\varphi_2)/A)^\frac{1}{2}. 
$$
Consequently, by possibly increasing $f$, we have that $b_p \leq A f(\mathcal I(\varphi_1,\varphi_1)/A)$, proving the claim. Comparing with \eqref{eq: a_ineq}, we thus have 
$$
|a_0-a_n|  \leq A f(\mathcal{I}(\varphi_1,\varphi_2)/A),
$$
what we wanted to prove. 	
\end{proof}

\subsection*{Approximation with rays of $C^{1, \bar 1}$ potentials}

The goal of this subsection is to prove Theorem \ref{thm: approximation_dp_intr}(ii):

\begin{theorem}\label{thm: main_C_11_approx_ray} Let $p \geq 1$. Suppose that $\{u_t\}_t \in \mathcal R^p_\o$ is such that $\mathcal K\{u_t\}<\infty$. Then there exists $\{v^k_t\}_t \subset \mathcal R^{1, \bar 1}_\o$ such that $v^k_t \searrow u_t, \ t \geq 0$, $d^c_p(\{v^k_t\}_t,\{u_t\}_t) \to 0$ and $\mathcal K\{v^k_t\} \to \mathcal K\{u_t\}$.
\end{theorem}

To argue this result, we need two auxilliary theorems, whose proof will be given at the end of the section. First we will need the following theorem, which will allow to obtain ``scaled" $C^{1, \bar 1}$ estimates along geodesic rays, via convexity:
\begin{theorem}\label{thm: He_C1bar1_est} Let $[0,1] \ni t \to u_t \in \mathcal H^{1, \bar 1}_\o$ be the $C^{1, \bar 1}$-geodesic connecting $u_0,u_1 \in \mathcal H_\omega^{1, \bar 1}$. Then there exists $B>0$, only depending on $(X,\omega)$ such that
$$[0,1] \ni t \to \textup{ess sup}_X (\log(n + \Delta_\omega u_t) - Bu_t) \in \Bbb R$$
is a convex function.
\end{theorem}

The proof of this theorem is obtained using the estimates developed in \cite{He15}. We will also need the following smoothing argument along bounded geodesic rays, relying on the regularizing property of the weak Monge-Amp\`ere flows, closely following the arguments of \cite{GZ17}:

\begin{theorem}\label{thm: bounded_ray_pot_approx} Let $B>0$ be from Theorem \ref{thm: He_C1bar1_est}, and $\{u_t\}_t \in \mathcal R^\infty_\omega$ with $\mathcal K\{u_t\} < \infty$ and $\sup_X u_t  =0, t \geq 0$. Then there exists $\alpha > 0$ depending on $(X,\omega)$ such that for all $s > 0$ and $j \in \mathbb{N}$ one can find $u^j_s \in \mathcal H_{\omega}$ satisfying the following conditions:\\
\noindent (i) $\{u_s^j\}_j$ is decreasing  and $u_s \leq  u_s^j \leq \alpha j 2^{-j}$,\\
\noindent (ii) $\sup_X \big(\log (n + \Delta_\omega  u^j_s) - B u^j_s\big) \leq \alpha 2^{j}(1+ s)$, \\
\noindent (iii) $d_1(u^j_s,u_s) \leq \alpha  2^{-j}  s + \alpha j 2^{-j}$,
\\
\noindent (iv) $\textup{Ent}(\omega^n,\omega_{u^j_s}^n) \leq \textup{Ent}(\omega^n,\omega_{u_s}^n)$.
\end{theorem}

\begin{proof}[Proof of Theorem \ref{thm: main_C_11_approx_ray}]  First we assume that $\{u_t\}_t \in \mathcal R^\infty_\o$ and $\sup_X u_t = 0, \ t \geq 0$. If $\{u_t\}_t$ is the constant ray then we are done, hence after rescaling we can also assume that $d_1(0,u_t) =t, \ t \geq 0$. Let $\{u_s^j\}_{s>0, j \in \mathbb{N}}$ be the potentials constructed as in Theorem \ref{thm: bounded_ray_pot_approx}.

Let us fix $j \in \mathbb{N}$ momentarily. Given $s>0$, by $[0,s] \ni t \to  u^{j,s}_t \in \mathcal H^{1, \bar 1}_\o$ we denote the $C^{1, \bar 1}$ geodesic connecting $ u^{j,s}_0 := 0$ and $ u^{j,s}_s := u^j_s$. Using condition (i) in  Theorem \ref{thm: bounded_ray_pot_approx} and the comparison principle for weak geodesics we get that
\begin{equation}\label{eq: bigger_ray_ineq}
u_t \leq  u^{j,s}_t \leq \frac{\alpha j2^{-j} t}{s}, \ \ t \in [0,s].
\end{equation}
Since $\{u^j_s\}_j$ is decreasing,  by the comparison principle for weak geodesics, $\{u^{j,s}_t\}_j$ is decreasing as well, for any $t \in [0,s]$. 

Given $t \in (0,s]$, by condition (ii) in  Theorem \ref{thm: bounded_ray_pot_approx} and Theorem \ref{thm: He_C1bar1_est} we have that
\begin{equation}\label{eq: C11_ray_est}
\frac{\textup{ess sup}_X (\log(1 + \frac{1}{n}\Delta_\omega  u^{j,s}_t) - B  u^{j,s}_t ) }{t} \leq   \frac{\sup_X (\log(1 + \frac{1}{n}\Delta_\omega  u^{j}_s) - B u^{j}_s ) }{s} \leq \alpha 2^j \Big(1 + \frac{1}{s}\Big).
\end{equation}
Finally, \eqref{eq: ChCh3_metric_convex} and condition (iii) in  Theorem \ref{thm: bounded_ray_pot_approx} implies that
\begin{equation}\label{eq: d1_ray_est}
\frac{d_1( u^{j,s}_t, u_t)}{t} \leq \alpha 2^{-j}  + \frac{\alpha}{s}, \ \ t \in [0,s].
\end{equation}
Fixing $t  >0$, \eqref{eq: bigger_ray_ineq} and \eqref{eq: C11_ray_est} gives that $\{u^{j,s}_t\}_{ \{s > t\}}$ is compact in the  $C^{1,\alpha}$ topology, implying existence of $v^{j}_t \in \mathcal H_{\omega}^{1, \bar 1}$ such that $\|v^{j}_t - u^{j,s}_t \|_{C^{1,\alpha}} \to 0 $ as $s \to \infty$ (after passing to a subsequence). Moreover, letting $s \to \infty$ in \eqref{eq: bigger_ray_ineq}, \eqref{eq: C11_ray_est} and \eqref{eq: d1_ray_est}, using Lemma \ref{lem: Laplace_aproxx}, we arrive at 
\begin{equation}\label{eq: ray_est_final}
u_t \leq v^{j}_t \leq 0, \ \ \ \frac{1}{t}\Big(\log(1 + \frac{1}{n}\Delta_\omega v^{j}_t) - B  v^{j}_t \Big) \leq   \alpha 2^j, \ \ \  \frac{d_1(v^{j}_t, u_t)}{t} \leq \alpha 2^{-j} , \ \ \ t \in (0,\infty).
\end{equation}
Using an Arzela--Ascoli type argument exactly the same way as in the proof of Theorem \ref{thm: K_approx}, after passing to a subsequence, we can assume that $\|v^j_t - u^{j,s}_t\|_{C^{1,\alpha}} \to 0$ for all $t >0$ at the same time, implying existence of $\{v^j_t\} \in \mathcal R^{1, \bar 1}$ for any $j \in \mathbb{N}$. By \eqref{eq: ray_est_final} we get that $d_1^c(\{v^j_t\},\{u_t\}) \to 0$ as $j \to \infty$. Remark \ref{rem: lemma_conv_not_monotone} implies that $d_p^c(\{v^j_t\},\{u_t\}) \to 0$, as desired.

Finally, since $\{u^{j,s}_t\}_j$ is decreasing for any $t \in [0,s]$, a diagonal Cantor process now implies that $\{v^j_t\}_j$ can be chosen to be decreasing for any $t >0$.

To show that $\mathcal K\{v^j_t\} \to \mathcal K\{u_t\}$, we first note that by \cite[Proposition 5.9]{ChCh3} we have $\mathcal K\{u_t\} \leq \liminf_j \mathcal K\{v^j_t\}$. Hence it is enough to show that $\mathcal K\{v^j_t\} \leq \mathcal K\{u_t\}+ f(\alpha 2^{-j})$ for any $j \in \Bbb N$, where $f: \Bbb R^+ \to \Bbb R^+$ is some continuous function with $f(0)=0$.

To prove this, we recall the Chen--Tian formula for the K-energy that extends to $\mathcal E^1_\o$ (see \cite[Theorem 1.2]{BDL17}):
\begin{equation}\label{eq: CT_formula}
\mathcal K(u) = \textup{Ent}(\omega^n,\omega_u^n) + \overline{S}I(u) - nI_{\textup{Ric }\o}(u), \ \ u \in \mathcal E^1_\o.
\end{equation}
Now, using conditions (iii) and (iv) in  Theorem \ref{thm: bounded_ray_pot_approx}  we can start writing:
\begin{flalign}\label{eq: K_ineq_u}
\frac{\mathcal K(u^j_s) - \mathcal K(u_s)}{s} &\leq \big|\overline{S}\big|\frac{d_1(u^j_s,u_s)}{s} + n\frac{I_{\textup{Ric }\o}(u^j_s)-I_{\textup{Ric }\o}(u_s)}{s} \nonumber\\
&\leq \frac{\alpha}{s} + \alpha 2^{-j} \big|\overline{S}\big| + n\frac{I_{\textup{Ric }\o}(u^j_s)-I_{\textup{Ric }\o}(u_s)}{s}.
\end{flalign}
We can suppose that $-C \omega \leq \textup{Ric }\omega  \leq C\omega$, and for the rest of the proof $C>0$ will denote a constant only dependent on $(X,\omega)$. Using condition (i) in  Theorem \ref{thm: bounded_ray_pot_approx} multiple times, we can start the following estimates
\begin{flalign*}
\frac{I_{\textup{Ric }\o}(u^j_s)-I_{\textup{Ric }\o}(u_s)}{s} &\leq  C \frac{\sum_j \int_X (u^j_s - u_s) \omega \wedge \omega_{u^j_s}^j \wedge \omega_{u_s}^{n-j-1}}{s}\\
& \leq  C \frac{\sum_j \int_X (u^j_s - u_s) \omega^n_{u^j_s/4 + u_s/4}}{s}\\
& \leq   C \frac{\sum_j \int_X (u^j_s - u_s) (\omega^n_{u^j_s/4 + u_s/4} - \omega_{u_s}^n)}{s} + C \frac{I(u^j_s) - I(u_s)}{s}\\
& \leq   f(\mathcal I(u^j_s,u_s)/s) + C\alpha 2^{-j}  + \frac{C\alpha}{s},
\end{flalign*}
where $f: \Bbb R^+ \to \Bbb R^+$ is a continuous function with $f(0)=0$, and in the last line we used Lemma \ref{lem: BBGZ I energy}. Together with \eqref{eq: K_ineq_u}, this inequality implies that 
$$\frac{\mathcal K(u^j_s)}{s} \leq \mathcal K\{u_t\} +  C \alpha 2^{-j} + f(\mathcal I(u^j_s,u_s)/s).$$
Letting $s \to \infty$, since $\mathcal K$ is convex and $d_1$-lsc, we obtain that $\mathcal K\{v^j_t\} \leq \mathcal K\{u_t\} + C \alpha 2^{-j}  + f(\alpha 2^{-j}),$  as desired, finishing the proof when $\{u_t\}_t \in \mathcal R^\infty_\o$.

Now let $\{u_t\}_t \in \mathcal R^p_\o$ with $\mathcal K\{u_t\}<\infty$.
By Theorem \ref{thm: K_approx}, there exists $\{u^{k}_t\} \in \mathcal R^{\infty}_\o$ such that $u^k_t \searrow u_t, \ t \geq 0$, $d_p^c(\{u^k_t\},\{u_t\}) \leq \frac{1}{2^k}$ and $\big| \mathcal K\{u_t\} - \mathcal K\{u^k_t\}\big| \leq \frac{1}{2^k}$. 

Let $\{u^{k,j}_s\}_j$ be the potentials of Theorem \ref{thm: main_C_11_approx_ray} associated to the rays $\{u^k_t\}_t$. By the construction of these potentials, elaborated in the proof of Theorem  \ref{thm: main_C_11_approx_ray}, it follows that $\{u^{k,j}_s \}_j$  is decreasing for any fixed $k \in \Bbb N$ and $s >0$. 
Using this, a diagonal Cantor process applied to the simultaneous approximation of each $\{u^k_t\}_t$ described above, yields rays $\{v^k_t\}_t \in \mathcal R^{1, \bar 1}$ such that $u^k_t\leq v^k_t$, $d_p^c(\{v^k_t,u^k_t\}) \leq \frac{1}{2^k}$, $\big| \mathcal K\{u^k_t\} - \mathcal K\{v^k_t\}\big| \leq \frac{1}{2^k}$, moreover (!) $\{v^{k}_t\}_k$ is decreasing for any fixed $t>0$. As $(\mathcal R^p_\omega, d_p^c)$ is complete, we obtain that $d_p^c(\{v^k_t\}_t,\{u_t\}_t) \to 0$ and $\mathcal K\{v^k_t\} \to \mathcal K\{u_t\}$, as desired.
\end{proof}

\begin{remark}\label{rem: C_1_1_slope} It follows from \eqref{eq: ray_est_final}, that the approximating rays $\{v^j_t\}_t \in \mathcal R^{1, \bar 1}$ in the previous theorem additionally satisfy the estimate:
$$\frac{1}{t}\textup{ess sup}_X \Big(\log(1 + \frac{1}{n}\Delta_\omega v^{j}_t) - B  v^{j}_t \Big)\leq \alpha 2^{j}, \ \ t >0, \ j \in \mathbb{N}.
$$
\end{remark}

\paragraph{The proof of Theorem \ref{thm: He_C1bar1_est}.}
First we recall some of the formalism of \cite{He15}. Given $u_0,u_1 \in \mathcal H_\o$, by $u^\varepsilon \in C^\infty([0,1] \times X)$ we denote the smooth $\varepsilon$-geodesic connecting $u_0,u_1$, i.e., $[0,1] \ni t \to u^\varepsilon_t \in \mathcal H_\o$ solves the following elliptic PDE on $[0,1] \times X$:
\begin{equation}
\label{eq: epsgeod}
\big(\ddot u_t^\varepsilon - |\nabla \dot u_t|^2_{\omega_{u_t^\varepsilon}}\big) \frac{\omega_{u^\varepsilon_t}^n}{\omega^n} = \varepsilon, \ \ \ u^\varepsilon_0 := u_0, \ u^\varepsilon_1 := u_1.
\end{equation}
Given that the complex Hessian of $u^\varepsilon$ is bounded on $[0,1] \times X$ \cite{Ch00}, one can take the limit $\varepsilon \to 0$, to obtain $u \in C^{1, \bar 1}([0,1] \times X)$, the $C^{1, \bar 1}$-geodesic connecting $u_0,u_1$: 
\begin{equation}\label{eq: H11geod}
[0,1] \ni t \to u_t \in \mathcal H_\o^{1, \bar 1}. 
\end{equation}
As shown in \cite[Theorem 1.1]{He15}, if one merely has $u_0,u_1 \in \mathcal H_\o^{1, \bar 1}$, the curve in \eqref{eq: H11geod} still exists, however it is not known if the total Laplacian of $u$ on $[0,1]\times X$ is bounded.

Let us denote the $\log$ of the left hand side of \eqref{eq: epsgeod} by $F(u^\varepsilon)$. Given a smooth function $h \in C^\infty([0,1] \times X)$, if $h$ attains its maximum at $(t,x) \in (0,1) \times X$, then ellipticity of \eqref{eq: epsgeod} gives that 
\begin{equation}\label{eq: ellipticity_ineq}
D F(u^\varepsilon)(h)(t,x):=\frac{d}{ds}\Big|_{s=0}F(u^\varepsilon + s h)(t,x) \leq 0.
\end{equation}

\begin{proof}[Proof of Theorem \ref{thm: He_C1bar1_est}] Let us first assume that $u_0,u_1 \in \mathcal H_\o$. Fix $(t,x) \in [0,1] \times X$ and $\varepsilon >0$ momentarily.  In \cite[page 339]{He15} (after equation (2.19)) it is shown that for some constants $B,C>1$,  dependent only on $(X,\omega)$, we have that
\begin{equation}\label{eq: He_ineq1}
D F(u^\varepsilon)(\log(n + \Delta_\omega u^\varepsilon_t) - B u^\varepsilon_t)(t,x) \geq \sum_{j=1}^n \frac{1}{1 + {(u^\varepsilon_t)}_{j\bar j}} - C,
\end{equation}
where we have used normal coordinates of $\omega$ at $x$ and $i \ddbar u^\varepsilon_t$ is assumed to be diagonal.
Additionally, fix $\delta >0$ and $g(t):= \delta t^2/2$. We also have  
\begin{equation}\label{eq: He_ineq2}
D F(u^\varepsilon)(g_\delta(t))(t,x) = \frac{\delta}{\ddot u_t - |\nabla u^\varepsilon_t|^2_{\omega_{u^\varepsilon_t}}}.
\end{equation}
Assume that $h_{\varepsilon,\delta}(t,x):=\log(n + \Delta_\omega u^\varepsilon_t) - B u^\varepsilon_t + g_\delta(t)$ is maximized at $(t,x) \in (0,1) \times X$. Then by \eqref{eq: ellipticity_ineq}, \eqref{eq: He_ineq1} and \eqref{eq: He_ineq2} we obtain at $(t,x)$ that
\begin{flalign}\label{eq: maxprinc}
0 & \geq DF(u^\varepsilon)(h_{\varepsilon,\delta}) \geq \sum_{j=1}^n \frac{1}{1 + {(u^\varepsilon_t)}_{j\bar j}} + \frac{\delta}{\ddot u_t - |\nabla u^\varepsilon_t|^2_{\omega_{u^\varepsilon_t}}} - C \nonumber \\
&\geq   (n+1)\bigg[\frac{\delta}{(1 + {(u^\varepsilon_t)}_{1, \bar 1}) \cdot \ldots \cdot (1 + {(u^\varepsilon_t)}_{n\bar n})(\ddot u_t - |\nabla u^\varepsilon_t|^2_{\omega_{u^\varepsilon_t}})}\bigg]^{\frac{1}{n+1}} - C \nonumber \\
&= (n+1)\bigg [\frac{\delta}{\varepsilon}   \bigg]^{\frac{1}{n+1}} - C. 
\end{flalign}
Thus, for $\varepsilon <  \delta (n+1)^{n+1}/C^{n+1}$ we get a contradiciton in the above inequality, implying that the maximum of   $h_{\varepsilon,\delta}$ can not be attained at $(t,x)$, an interior point of $[0,1] \times X$. In particular, we have that
$$\sup_X  h_{\varepsilon,\delta}(t,x) \leq \max(\sup_X h_{\varepsilon,\delta}(0,x), \sup_X h_{\varepsilon,\delta}(1,x)), \ \ t \in [0,1], \ \varepsilon <  \delta (n+1)^{n+1}/C^{n+1}.$$
Letting $\varepsilon \searrow 0$ and $\delta \searrow 0$ thereafter, via Lemma \ref{lem: Laplace_aproxx} we arrive at 
$$\textup{ess sup}_X  h_{0,0}(t,x) \leq \max\big(\sup_X h_{0,0}(0,x), \sup_X h_{0,0}(1,x)\big), \ \ t \in [0,1],$$
motivating the introduction of $M_{u_0,u_1}(t):=\textup{ess sup}_X(\log(n + \Delta_\omega (u_t)) - B u_t)$. Indeed, we can simply write 
\begin{equation}\label{eq: Mu_max_est}
M_{u_0,u_1}(t) \leq \max(M_{u_0,u_1}(0),M_{u_0,u_1}(1)), \ t \in [0,1].
\end{equation}
Next we observe that \eqref{eq: Mu_max_est} also holds in case we merely have $u_0,u_1 \in \mathcal H_\omega^{1, \bar 1}$. Indeed, we pick sequences $u^j_0 \searrow u_0$ and $u^j_1 \searrow u_1$, as in Proposition \ref{prop: Laplace_realiz_aproxx}. Then we apply \eqref{eq: Mu_max_est} to $M_{u_0^j,u_1^j}$ and  the comparison principle (\cite[Theorem 21]{Bl13}) together with Lemma \ref{lem: Laplace_aproxx} gives \eqref{eq: Mu_max_est} for $u_0,u_1 \in \mathcal H_\omega^{1, \bar 1}$.

To finish, we show that $M_{u_0,u_1}(t)$ is actually convex. Let $a,b \in [0,1]$. Then $t \to v_t:= u_{a + t (b-a)} + \frac{1}{B}(M_{u_0,u_1}(a) + t(M_{u_0,u_1}(b)-M_{u_0,u_1}(a))$ is the $C^{1, \bar 1}$ geodesic connecting $v_0:= u_a + M_{u_0,u_1}(a)/B$ and $v_1:= u_b + M_{u_0,u_1}(b)/B$. Applying \eqref{eq: Mu_max_est} to $v_0,v_1$ we arrive at
\begin{flalign*}
M_{u_0,u_1}(a + t(b-a)) - M_{u_0,u_1}(a)  - t(M_{u_0,u_1}(b) & -M_{u_0,u_1}(a))  = M_{v_0,v_1}(t) \\
&\leq \max(M_{v_0,v_1}(0),M_{v_0,v_1}(1)) =  0,
\end{flalign*}
hence $t \to M_{u_0,u_1}(t)$ is convex, as desired.
\end{proof}

\paragraph{The proof of Theorem \ref{thm: bounded_ray_pot_approx}.} In the proof of Theorem \ref{thm: bounded_ray_pot_approx} we will use the formalism of \cite{GZ17} adapted to our context. Fixing $\varphi_0 \in \mathcal{E}^1_{\omega}$ with $\sup_X \varphi_0=0$, we consider the following parabolic PDE on $[0,\infty) \times X$ with  initial data given by $\varphi_0$:
\begin{equation}\label{eq: GZ_parabolic}
\frac{d}{dt} \varphi_t = \log \bigg[ \frac{(\omega + i\ddbar \varphi_t)^n}{\omega^n} \bigg].
\end{equation}
To avoid cumbersome notation, we will denote $t$-derivatives by dots throughout this paragraph.
 As shown in \cite{GZ17}, $(t,x) \to \varphi_t(x)$ is smooth on $(0,\infty)\times X$. The initial condition simply means that $d_1(\varphi_t,\varphi_0)\to 0$, as $t \to 0$ \cite[Section 5.2.2]{GZ17}. In case $\varphi_0 \in \mathcal H_\o$, we actually have that $\|\varphi_t - \varphi_0\|_{C^\infty} \to 0$ as $t \to 0$. Moreover it is shown in \cite[Theorem B]{DiLu17} that if  $\varphi_0 \in \mathcal E^1_\o$  and $\varphi_{0}^j \in \mathcal H_{\omega}$ converges in $L^1(X,\omega^n)$ to $\varphi_0$, then for any $t >0$ we have that $\|\varphi_{t}^j - \varphi_{t}\|_{C^\infty} \to 0$, where $\{\varphi^j_{t}\}_j$ are the smooth solutions to \eqref{eq: GZ_parabolic} with initial data $\varphi_{0}^j$. All this implies that the apriori estimates and maximum principles developed in \cite[Section 2]{GZ17} for smooth initial data, also apply for initial data in $\mathcal E^1_\o$, as above (for our applications $\varphi_0$ will be actually bounded). 
 
For the remainder of this paragraph we pick a  small constant $\lambda>0$  depending only on $(X,\omega)$ such that $\int_X e^{-2\lambda \phi} \omega^n$ is uniformly bounded for all $\phi \in \PSH(X,\omega)$ normalized by $\sup_X \phi=0$ (see \cite[Proposition 2.1]{Ti87}, \cite{Zer01}). 

Let $v$ be the unique continuous $\omega$-psh function such that 
\begin{equation}\label{eq: v_s_def}
\omega_{v}^n : =  e^{\lambda v -\lambda \varphi_0 - n \log \lambda} \omega^n. 
\end{equation}
By our choice of $\lambda,$ it follows from \cite{Ko98, BBGZ13} (or much more generally \cite[Theorem 5.3]{DDL4}) that $v$ is uniformly bounded by a constant depending only on  $(X,\omega)$.

\begin{lemma}\label{lem: 2.2(i)} With $\lambda\in (0,1)$ and $v$ as above, we have that
\begin{equation}\label{eq: lemmaest}
(1-\lambda t) \varphi_0  + \lambda t v + n(t \log t - t) \leq \varphi_t \leq 0, \ \ t \in [0,1].
\end{equation} 
\end{lemma}

\begin{proof}
Let $\psi_t := (1-\lambda t) \varphi_0  + \lambda t v + n(t \log t - t), \ t \in [0,1]$. The following hold:
$$\dot \psi_t = \lambda (v - \varphi_0) + n \log t  = \log \bigg( \lambda^n t^n \cdot   \frac{\omega^n_{v}}{\omega^n}\bigg) \leq  \log \bigg( \frac{\omega^n_{\psi_t}}{\omega^n}\bigg).$$
This implies that $l \to \psi_t$ is a subsolution to \eqref{eq: GZ_parabolic}, and an application of the maximum principle \cite[Corollary 2.2]{GZ17} yields the first inequality in \eqref{eq: lemmaest}. The second inequality follows from \cite[Lemma 2.3]{GZ17}.
\end{proof}
\noindent Simplifying \eqref{eq: lemmaest}, we actually obtain that:
\begin{equation}\label{eq: lem_cor} \varphi_0   \leq \varphi_t + Ct - C t \log t,  \ \ t \in [0,1].
\end{equation}
for some constant $C>0$ dependent on $(X,\omega)$. This can be taken one step further, as we now describe:
\begin{corollary}
\label{cor: flow decreasing}
There exists a constant $C>1$ depending on $(X,\omega)$ such that $w_{t} \geq w_{t/2}$ for any $t\in [0,1]$, where 
\begin{equation}\label{eq: w_t_def}
w_t := \varphi_t + Ct - C t\log t. 
\end{equation}
\end{corollary}

\begin{proof}
Fixing $s \in (0,1)$, we apply \eqref{eq: lem_cor} to the flow $t\mapsto \varphi_{s/2 + t}$, starting from $\varphi_{s/2}$. By \eqref{eq: lemmaest} and \eqref{eq: lem_cor}  we have that $\|e^{-\lambda \varphi_{s/2}}\|_{L^2}$ is controlled by $\|e^{-\lambda \varphi_0}\|_{L^2}$ which is uniformly bounded by a constant depending on $(X,\omega)$. Hence for $t:=s/2 \in [0,1]$ in \eqref{eq: lem_cor} we have
$$
\varphi_{s} \geq \varphi_{s/2} - C s/2 + C (s/2) \log (s/2),
$$
where $C$ only depends on $(X,\omega)$.
Thus, after possibly increasing $C>0$, the function $w_t := \varphi_t + Ct -Ct \log t$ satisfies $w_t \geq w_{t/2}$, $t\in [0,1]$. 
\end{proof}

We also point out the following simple monotonicity result:

\begin{lemma} \label{lem: 2.2(iv)} The map $[0,\infty) \ni t \to \textup{Ent}(\omega^n,\omega_{\varphi_t}^n) \in \Bbb R$ is decreasing.
\end{lemma}

\begin{proof} First let us assume that $\varphi_0 \in \mathcal H_\o$ in \eqref{eq: GZ_parabolic}. For $t \geq 0 $, we can start by computing
\begin{flalign*}
\frac{d}{dt}\textup{Ent}(\omega^n,\omega_{\varphi_t}^n) &= \frac{d}{dt} \int_X \dot  \varphi_t (\omega + i\ddbar \varphi_t)^n\\
&=\int_X \ddot \varphi_t (\omega + i\ddbar \varphi_t)^n - \int_X | \nabla \dot  \varphi_t|^2_{\omega_{\varphi_t}} (\omega + i\ddbar \varphi_t)^n\\
&=\int_X \big( \Delta_{\omega_{\varphi_t}} \dot  \varphi_t \big) (\omega + i\ddbar \varphi_t)^n - \int_X | \nabla \dot  \varphi_t|^2_{\omega_{\varphi_t}} (\omega + i\ddbar \varphi_t)^n\\
&= - \int_X | \nabla \dot  \varphi_t|^2_{\omega_{\varphi_t}} (\omega + i\ddbar \varphi_t)^n  \leq 0.
\end{flalign*}
Consequently, $t \to \textup{Ent}(\omega^n,\omega_{\varphi_t}^n)$ is decreasing on $[0,\infty)$, when $\varphi_0 \in \mathcal H_\o$. 

For general $\varphi_0 \in \mathcal{E}^1_{\omega}$, let $\varphi_{0}^j \in \mathcal H_\o$ be such that $d_1(\varphi_0^j,\varphi_0) \to 0$ and $\textup{Ent}(\omega^n, \omega_{\varphi_{0}^j}^n) \to \textup{Ent}(\omega^n, \omega_{\varphi_0}^n)$ (such sequence exists by \cite[Theorem 1.3]{BDL17}). Fixing $t>0$, by \cite[Theorem B]{DiLu17} we have that $\varphi^j_{t} \to_{C^\infty} \varphi_{t}$, hence we can  conclude that
$$\textup{Ent}\left(\omega^n,\omega_{\varphi_{t}}^n\right) = \lim_j \textup{Ent}(\omega^n,\omega_{\varphi^j_{t}}^n) \leq \lim_j \textup{Ent}(\omega^n,\omega_{\varphi_{0}^j}^n) =\textup{Ent}(\omega^n, \omega_{\varphi_0}^n),$$
finishing the proof.
\end{proof}

For the remainder of this paragraph, let $\{u_t\}_t \in \mathcal R^\infty$ with $\sup_X u_t =0, \  t\geq 0$ and $\mathcal K\{u_t\} < + \infty$, as in the statement of Theorem \ref{thm: bounded_ray_pot_approx}. Since $\sup_X u_s =0, \ s \geq 0$, by the weak $L^1$-compactness of $\textup{PSH}(X,\omega)$ we have that $u_s \searrow u_\infty \in \textup{PSH}(X,\omega)$.  

We fix $s>0$ for the remainder of this paragraph, and we construct the sequence $u_s^j$ as follows. For each $j$ we define 
$$u_s^j := w_{s,2^{-j}},$$ 
where $w_{s,t} \in \mathcal H_\o$ is constructed in \eqref{eq: w_t_def} with respect to the flow $t \to \varphi_{s,t}$, starting from $\varphi_{s,0}:=u_s$.  The estimate of Corollary \ref{cor: flow decreasing}, together with \eqref{eq: lem_cor} yields the condition (i) in  Theorem \ref{thm: bounded_ray_pot_approx} for $\alpha:= 2C$. Condition (iv) follows automatically from Lemma \ref{lem: 2.2(iv)}.

Next we address condition (ii) in Theorem \ref{thm: bounded_ray_pot_approx}, which is closely related to  \cite[Corollary 4.5]{GZ17}:
\begin{lemma}\label{lem: 2.2(ii)} We have that 
\begin{equation}\label{eq: (ii)lemma}
\sup_X \big(\log (n + \Delta_\omega  u^j_s) - B u^j_s\big) \leq \alpha 2^j (1+ s), \ \ j \in \mathbb{N}, \ s>0.
\end{equation}
\end{lemma}
\begin{proof} From \cite[Corollary 4.5]{GZ17} we obtain that for any $j \in \Bbb N$ and $s >0$ we have 
\begin{equation}\label{eq: GZCor4.5}
\frac{1}{2^j} \log(n + \Delta_\omega {u^j_s}) = \frac{1}{2^j} \log(n + \Delta_\omega {\varphi_{s,{2^{-j}}}})  \leq  C(\textup{osc}_X \varphi_{s,{2^{-j-1}}} + 1),
\end{equation}
where $C>0$ only depends on $(X,\omega)$. Using \eqref{eq: lem_cor} we have that 
$$\textup{osc}_X \varphi_{s,{2^{-j-1}}} \leq -\inf_X \varphi_{s,0} +\alpha=-\inf_X u_s +\alpha.$$ By \cite[Theorem 1]{Da17} we have that $\inf_X u_s = m_{\{u_t\}} s$ for some constant $m_{\{u_t\}} \leq 0$. Consequently \eqref{eq: (ii)lemma} follows after putting the above together with condition (i) in Theorem \ref{thm: bounded_ray_pot_approx} (and possibly increasing the value of $\alpha>0$).
\end{proof}

\noindent Next we address condition (iii) in Theorem \ref{thm: bounded_ray_pot_approx}:

\begin{lemma}\label{lem: 2.2(iii)} We have that $d_1(u^j_s,u_s) \leq \alpha 2^{-j}    s + \alpha j 2^{-j}$ for any $j \in  \mathbb{N}, s>0$.
\end{lemma}
\begin{proof} For the flow $t\mapsto \varphi_{s,t}$, using the equation \eqref{eq: GZ_parabolic}, we can write 
\begin{flalign*}
I(\varphi_{s,t}) - I(\varphi_{s,0}) & = \int_0^t \frac{d}{dl}I(\varphi_{s,l}) dl = \int_0^t  \textup{Ent}(\omega^n, \omega_{\varphi_{s,l}}^n) dl\\  
&\leq \textup{Ent}(\omega^n, \omega_{\varphi_{s,0}}^n) t=\textup{Ent}(\omega^n, \omega_{u_s}^n) t, 
\end{flalign*}
where we have used  Lemma \ref{lem: 2.2(iv)}. Recall that for $u^j_s: =w_{s,2^{-j}}$, due to property (i) we can continue:
\begin{flalign*}
d_1(u^j_s,u_s) &=  I(u^j_s) - I(u_s) \leq I(\varphi_{s,{2^{-j}}}) - I(\varphi_{s,0}) + \alpha j 2^{-j}    \leq \textup{Ent}(\omega^n, \omega_{u_s}^n) 2^{-j} + \alpha j 2^{-j}.
\end{flalign*}
 After invoking Lemma \ref{lem: Ent_sub_linear} below, and possibly adjusting $\alpha>0$ again, the proof is finished.
\end{proof}

\noindent As promised above, we argue that along $\{u_t\}_t$ the entropy has  sublinear growth:

\begin{lemma}\label{lem: Ent_sub_linear} There exists $C:= C(\{u_t\}_t)>0$ such that $\textup{Ent}(\omega^n, \omega_{u_t}^n) \leq C t, \ t \geq 0.$
\end{lemma}

\begin{proof} Let $D > \mathcal K\{u_t\}$. By the Chen--Tian formula for the extended K-energy \eqref{eq: CT_formula} we obtain that
\begin{flalign*}
\textup{Ent}(\omega^n,\omega^n_{u_t}) & \leq Dt - \bar{S}I(u_t) + n I_{\textup{Ric }\o}(u_t)  \leq  Ct + C d_1(0,u_t) + C d_1(0,u_t) \leq C t,
\end{flalign*}
where we have used \cite[Proposition 2.5]{DH17} in the second estimate.
\end{proof}

\section{Applications to geodesic stability}

First we point out how the $L^1$ version of Conjecture \ref{conj: geod_stability} can be derived from \cite{ChCh2,ChCh3} and \cite[Theorem 4.7]{Da18}. As alluded to in the introduction, the argument is implicitly contained in \cite{ChCh2,ChCh3}, but we provide a short proof here as this result is not explicitly stated in that paper. Recall that $G = \textup{Aut}_0(X,J)$, and for the definition of $G$-calibrated rays we refer back to the introduction.

\begin{theorem}[$L^1$ uniform geodesic stability]\label{thm: L1geod_stability} Let $(X,\omega)$ be a compact K\"ahler manifold. Then the following are equivalent:\vspace{0.1cm}\\
(i) There exists a csck metric in $\mathcal H_\o$.\vspace{0.1cm}\\
(ii) There exists $\delta >0$ such that $\mathcal K\{u_t\} \geq \delta \limsup_t\frac{d_{1,G}(G 0,G u_t)}{t}$ for all  $\{u_t\}_t \in \mathcal R^1$.\vspace{0.1cm}\\
(iii) $\mathcal K$ is $G$-invariant and there exists $\delta >0$ s.t. $\mathcal K\{u_t\} \geq \delta d_1(0,u_1)$  for all $G$-calibrated geodesic rays $\{u_t\}_t \in \mathcal R^1$.
\end{theorem}
We recall that $\mathcal{R}^p/\mathcal R^{1,\bar 1}, \ p \in [1,\infty]$ is the set of rays $\{u_t\}_t \in \mathcal R^p_\o/\mathcal R^{1,\bar 1}_\o$  normalized by the condition $I(u_t)=0, \ t \geq 0$. 

\begin{proof} By \cite[Theorem 1.5]{ChCh2}, the conditions of  \cite[Theorem 4.7]{Da18} are satisfied. Indeed, it was pointed out in \cite[Theorem 10.1]{DR17} that all the conditions (A1)-(A4) and (P1)-(P6) hold with the exception of (P3), which is exactly the content of \cite[Theorem 1.5]{ChCh2}.

After comparing with the conclusion of \cite[Theorem 4.7]{Da18}, we only need to argue that condition (ii) implies that $\mathcal K$ is $G$-invariant. However we notice that (ii) implies that $(X,\omega)$ is $L^1$-geodesically semistable, in the sense that, $\mathcal K\{u_t\} \geq 0$ for any $\{u_t\}_t \in \mathcal R^1_{\omega}$. Now \cite[Lemma 4.1]{ChCh3} implies that $\mathcal K$ is $G$-invariant as desired.
\end{proof}

To show that Theorem \ref{thm: Linftygeod_stability_intr} holds, we argue in the next two results that conditions (ii) and (iii) in the previous theorem are equivalent with their  $C^{1, \bar 1}$ version:
\begin{theorem}\label{thm: L1Linft_geod_eqv} Let $(X,\omega)$ be a compact K\"ahler manifold. Then the following are equivalent:\vspace{0.1cm}\\
(i) There exists $\delta >0$ such that $\mathcal K\{u_t\} \geq \delta \limsup_t\frac{d_{1,G}(G 0,G u_t)}{t}$ for all  $\{u_t\}_t \in \mathcal R^1$.\vspace{0.1cm}\\
(ii) There exists $\delta >0$ such that $\mathcal K\{u_t\} \geq \delta \limsup_t\frac{d_{1,G}(G 0,G u_t)}{t}$ for all  $\{u_t\}_t \in \mathcal R^{1, \bar 1}$.
\end{theorem}
\begin{proof} We only need to argue that (ii)$\Rightarrow$(i). Let $\{u_t\}_t \in \mathcal R^1$. We can assume that $\mathcal K\{u_t\} < \infty$, otherwise there is nothing to prove. 

We pick $\{u^k_t\}_t \in \mathcal R^{1,\bar{1}}_\omega$, as in Theorem \ref{thm: main_C_11_approx_ray}. We notice that $|I(u_t)-I(u^k_t)| \to 0$ as $k \to \infty$ for fixed $t \geq 0$, hence by subtracting a linear term form each $\{u^k_t\}_t$ we can assume that $\{u^k_t\}_t \in \mathcal R^{1,\bar 1}$ with  $\mathcal K\{u^k_t\}\to \mathcal K\{u_t\}$ and $d^c_1(\{u^k_t\}_t,\{u_t\}_t) \to 0$ still holding. Moreover, we have the following sequence of inequalities:
\begin{flalign*}
\limsup_t \frac{\big| d_{1,G}(G 0,G u_t) - d_{1,G}(G 0,G u^k_t)\big|}{t} & \leq \limsup_t\frac{d_{1,G}(G u_t,G u^k_t)}{t}\\
& \leq \limsup_t\frac{d_{1}(u_t,u^k_t)}{t}=d_1^c(\{u_t\}_t,\{u_t^k\}_t).
\end{flalign*}
Since the last term converges to zero as $k \to \infty$, we obtain that $\limsup_t{d_{1,G}(G 0,G u^k_t)}/{t}$ converges to $\limsup_t{d_{1,G}(G 0,G u_t)}/{t}$, as desired.
\end{proof}

\begin{theorem}\label{thm: L1Linft_geod_calibrated_eqv}Let $(X,\omega)$ be a compact K\"ahler manifold. Then the following are equivalent:\vspace{0.1cm}\\
(i) $\mathcal K$ is $G$-invariant and there exists $\delta >0$ s.t. $\mathcal K\{u_t\} \geq \delta  d_1(0,u_1)$  for all $G$-calibrated geodesic rays $\{u_t\}_t \in \mathcal R^1$. \vspace{0.1cm}\\
(ii) $\mathcal K$ is $G$-invariant and there exists $\delta >0$ s.t. $\mathcal K\{u_t\} \geq \delta  d_1(0,u_1)$  for all $G$-calibrated geodesic rays $\{u_t\}_t \in \mathcal R^{1, \bar 1}$.
\end{theorem}

\begin{proof} We only need to argue that (ii)$\Rightarrow$(i). Let $\{u_t\}_t \in \mathcal R^1$, $G$-calibrated and non-constant. We can assume that $\mathcal K\{u_t\} < \infty$, otherwise there is nothing to prove. 

Using Theorem \ref{thm: main_C_11_approx_ray}, we pick $\{u^k_t\}_t \in \mathcal R^{1, \bar 1}_\omega$ such that $d_1^c(\{u^k_t\},\{u_t\}) \to 0$ and $\mathcal K\{u^k_t\} \to \mathcal K\{u_t\}$. By adjusting with small constants, we can assume that $\{u^k_t\}_t \in \mathcal R^{1, \bar 1}$, and neither of these rays is constant.
 Unfortunately $\{u^k_t\}_t$ may not be $G$-calibrated, and the bulk of the proof consists of finding a new sequence $ \{\tilde u^k_t\}_t \in\mathcal R^{1, \bar 1}$ that satisfies this property.

For any $k \geq 1$ and $t \geq 1$, let $g^{k}_t \in G$ such that 
\begin{equation}\label{eq: g_k_t_ineq}
d_1(0,g^k_t.u^k_t) \geq d_{1,G}(0,u^k_t) \geq d_1(0,g^k_t.u^k_t) - \frac{1}{t}.
\end{equation} 
The following estimates will be used later:
\begin{flalign}\label{eq: g_kk_est}
d_1(0,g^k_t. u^k_t) &\geq d_1(0,g^k_t. u_t) - d_1(g^k_t.u^k_t, g^k_t.u_t)\nonumber \\
& \geq d_{1,G}(G.0,G.u_t) - d_{1}(u_t,u^k_t) \nonumber \\
& \geq d_1(0,u^k_t) - 2 d_1(u_t,u^k_t)\\
& \geq d_1(0,u^k_t) - 2t d^c_1(\{u_t\}_t,\{u_t^k\}_t), \nonumber 
\end{flalign}
where in the first line we have used the triangle inequality;  in the second line we have used the definition of $d_{1,G}$ and the fact that $G$ acts on $\mathcal{E}^1_0$ by $d_1$-isometry (see \cite[Lemma 5.9]{DR17});  in the third line we have used the triangle inequality and  that $\{u_t\}_t$ is $G$-calibrated; in the last line we have used that $(0,+\infty) \ni t\mapsto d_1(u_t^k,u_t)/t$ is increasing (see \cite{BDL16}).   

Let $[0,t] \ni l \to \rho^{k,t}_l \in \mathcal E^1_\o$ be the finite energy geodesic connecting $0$ and $g^k_t. u^k_t$. From \eqref{eq: g_k_t_ineq} and \cite[Lemma 4.9]{Da18} it follows that
\begin{equation}\label{eq: rho_k_t_l_ineq}
d_1(0,\rho^{k,t}_l) \geq d_{1,G}(G.0,G.\rho^{k,t}_l) \geq d_1(0,\rho^{k,t}_l) - \frac{1}{t}, \ \ \ l \in [0,t].
\end{equation} 
Using $G$-invariance, convexity of $\mathcal K$, and that $\mathcal K(0)=0$, for any $l \in [0,t]$ we have that
\begin{flalign}\label{eq: K-first_ineq}
\frac{\mathcal K(\rho^{k,t}_l)}{l} \leq \frac{\mathcal K(g^k_t.u^k_t)}{t} = \frac{\mathcal K(u^k_t)}{t}\leq \mathcal K\{u_t^k\}.
\end{flalign}
Due to \cite[Corollary 4.8]{BDL17}, after possibly selecting a subsequence $t_j \to \infty$, there exists $\tilde u^k_l \in \mathcal E^1_{\omega}$ for any $l > 0$, such that $d_1(\tilde u^k_l,\rho_{l}^{k,t_j}) \to 0$. After taking the limit in \eqref{eq: rho_k_t_l_ineq}, due to \cite[Proposition 4.3]{BDL17} we find that $\{\tilde u^k_t \}_t \in \mathcal R^1$ is $G$-calibrated. Moreover, due to \eqref{eq: g_kk_est}, there exists $k_0$ such that $\{\tilde u^k_t \}_t$ is not the constant ray for $k \geq k_0$.

Next we argue that $\{\tilde u^k_t \}_t \in \mathcal R^{1, \bar 1}$. To start,  for $t \geq 1$ using \eqref{eq: g_k_t_ineq} and the fact that $G$ acts by $d_1$-isometries (see \cite[Lemma 5.9]{DR17}), we get that 
\begin{flalign}\label{eq: d_1_g_est}
d_1(0,g^k_t.0) &= d_1(0,(g^k_t)^{-1}.0) \leq d_1(u^k_t,(g^k_t)^{-1}.0) + d_1(u^k_t,0)
 \nonumber\\
&= d_1(g^k_t. u^k_t,0) + d_1(u^k_t,0) \leq 2 d_1(u^k_t,0) + \frac{1}{t} \leq 2 d_1(u^k_t,0) + 1.
\end{flalign}

Next, let $B>0$ as in the statement of Theorem \ref{thm: He_C1bar1_est}. Using Lemma \ref{lem: action_est_d_1} below, we have the following estimates
$$\max \Big(\sup_X |g^k_t.0|, \   \sup_X \log |\nabla g^k_t|_\o, \  \sup_X \big(\log(n + \Delta_\omega (g^k_t.0) ) - B g^k_t.0\big) \Big) \leq C d_1(g^k_t.0,0) + C.$$

Recall that  $g_t^k.u^k_t = (g_t^k)^*u^k_t + g^k_t.0$ (see \cite[Lemma 5.8]{DR17}). In particular, \cite[Theorem 1]{Da13}, Remark \ref{rem: C_1_1_slope} and \eqref{eq: d_1_g_est} give that
$$\sup_X |g_t^k.u^k_t| \leq \sup_X |u^k_t| + \sup_X |g^k_t.0| \leq C t + 2C d_1(u^k_1,0) t + C \leq C t + C,$$
$$\sup_X \big(\log(n + \Delta_\omega (g^k_t.u^k_t) ) - B g^k_t.u^k_t\big) \leq C t + 2C d_1(u^k_1,0) t + C \leq C t + C,$$
where $C$ depends on $k$ but not on $t \geq 1$!  Using  \cite[Theorem 1]{Da13} and   Theorem \ref{thm: He_C1bar1_est}, we find that $\sup_X |\rho^{k,t}_l| \leq C l + C$ and $\sup_X \big(\log(n + \Delta_\omega \rho^{k,t}_l ) - B \rho^{k,t}_l\big)\leq Cl + C$ for any $l \in [0,t]$. Lastly, letting $t_j \to \infty$, we arrive at  $\sup_X |\tilde u^k_l| \leq C l + C$ and $\sup_X \big(\log(n + \Delta_\omega \tilde u^k_l ) - B \tilde u^k_l\big)\leq Cl + C$ for any $l \geq 0$, what we wanted to argue.

Due to the fact that $\mathcal K$ is $d_1$-lsc, $G$-invariant and convex, similar to \eqref{eq: K-first_ineq}, we find that for all $l > 0$ and $k \geq k_0$ we have:
\begin{flalign}\label{eq: K-second_ineq}
\frac{\mathcal K(\tilde u^k_l)}{d_1(0,\tilde u^k_l)} &\leq \liminf_{t_j \to \infty } \frac{\mathcal K(\rho^{k,t_j}_l)}{d_1(0,\rho^{k,t_j}_l)} \leq  \liminf_{t_j \to \infty } \frac{\mathcal K(g^k_{t_j}. u^k_{t_j})}{d_1(0,g^k_{t_j}. u^k_{t_j})} =  \liminf_{t_j \to \infty } \frac{\mathcal K( u^k_{t_j})}{d_1(0,g^k_{t_j}. u^k_{t_j})} \nonumber \\
& \leq \liminf_{t_j \to \infty}\frac{\mathcal K(u^k_{t_j})}{d_1(0,u^k_{t_j}) - 2t_j d_1^c(\{u^k_t\}_t,\{u_t\}_t)} \nonumber \\
&= \liminf_{t_j \to \infty} \bigg[\frac{\mathcal K(u^k_{t_j})}{d_1(0,u^k_{t_j})} \cdot \frac{d_1(0,u^k_{t_j})}{d_1(0,u^k_{t_j}) - 2t_j d_1^c(\{u^k_t\}_t,\{u_t\}_t)} \bigg] \nonumber \\
& \leq \frac{\mathcal K\{u^k_t\}}{d_1(0,u^k_1)} \cdot \frac{d_1(0,u^k_1)}{d_1(0,u^k_1)- 2 d_1^c(\{u^k_t\}_t,\{u_t\}_t)},
\end{flalign}
where in the second line we have used \eqref{eq: g_kk_est}, and all the denominators are non-zero since $\{\tilde u^k_t\}_t$ and $\{u^k_t\}_t$ are non-constant for $k \geq k_0$.

Finally, we use that (ii) holds for $\{\tilde u^k_t\}_t \in \mathcal R^{1, \bar 1}$. Consequently, after letting $l,t \to \infty$ in \eqref{eq: K-second_ineq}, we arrive at
$$\delta \leq \frac{\mathcal K \{\tilde u^k_t\}}{d_1(0, \tilde u^k_1)} \leq \frac{\mathcal K\{u^k_t\}}{d_1(0,u^k_1)} \cdot\frac{d_1(0,u^k_1)}{d_1(0,u^k_1)- 2 d_1^c(\{u^k_t\}_t,\{u_t\}_t)}.$$
Letting $k \to \infty$, we now obtain that $\delta \leq \frac{\mathcal K \{u_t\}}{d_1(0,u_1)},$ finishing the proof.
\end{proof}

\begin{lemma}\label{lem: action_est_d_1}Let $(X,\omega)$ be a compact K\"ahler manifold. There exists $C:=C(X,\omega)>0$ such that for all $g \in G$ we have that $\sup_X |g.0| \leq Cd_1(0,g.0) + C$ and $\sup_X \log(n + \Delta_\omega (g.0) ) \leq Cd_1(0,g.0) + C$.
\end{lemma}

The Laplacian estimate from this lemma is equivalent with the following estimate for the gradient $\nabla g$, as a self map of $X$:
\begin{equation}\label{eq: g_nabla_d_1_est}
\sup_X |\nabla g|^2_\omega \leq e^{C d_1(0,g.0) + C}, \ \ g \in G.
\end{equation}

The desired Laplacian estimate of the lemma can be extracted from the arguments of \cite{CDH17}, as we now elaborate.

\begin{proof} 
Fix $g \in G$. Using \cite[Lemma 5.8]{DR17}, and the fact that $({g^{-1}})^* \big(g^* \omega\big) = \omega$, we obtain that $0 =  g^{-1}. (g. 0) = g^{-1}.0 + ({g^{-1}})^* \big( g.0\big)$. In particular, we have that
\begin{equation}\label{eq: g_inv_g_id}
-\inf_X g.0 =   \sup_X {g^{-1}}.0.
\end{equation}
Due to \cite[Corollary 4]{Da15}  and \cite[Lemma 3.45]{Da18} we have that 
$$
0 \leq \sup_X g.0 \leq \int_X g.0 \, \omega^n  + C \leq C d_1(0,g.0) + C,
$$
and
$$0 \leq \sup_X g^{-1}.0 \leq \int_X g^{-1}.0\, \omega^n \leq C d_1(0,g^{-1}.0) + C.$$
Since $d_1(0,g^{-1}.0) = d_1(g.0,0)$, putting the above together with  \eqref{eq: g_inv_g_id},\ one of the desired estimates follows:
\begin{equation}\label{eq: C_0_g_est}
\sup_X |g.0| \leq Cd_1(0,g.0) + C.
\end{equation}

Now we address the Laplacian estimate. To start, we note that there exists $C:=C(X,\omega)>0$ such that $-C \omega \leq \textup{Ric } \omega \leq C \omega$. Pulling back by $g$ we obtain that $\textup{Ric } \omega_{g.0} \leq C \omega_{g.0}$. We introduce $F_g := \log \Big( \frac{\omega_{g.0}^n}{\o^n}\Big)$. We obtain that 
$$i\partial \bar \partial F_g = \textup{Ric } \o - \textup{Ric } \o_{g.0} \geq -C \omega - C \omega_{g.0}.$$
In particular, $\frac{1}{2} g.0  + \frac{1}{2C} F_g \in \textup{PSH}(X,\omega)$, implying that
$$
\sup_X \left( \frac{1}{2} g.0 +  \frac{1}{2C} F_g \right) \leq C+ \int_X \left (\frac{1}{2} g.0 +  \frac{1}{2C} F_g \right )\omega^n \leq \frac{1}{2}d_1(0,g.0) + C.
$$
Here, we used Jensen's inequality to obtain $\int_X F_g \omega^n \leq 0$. 
Using \eqref{eq: C_0_g_est} we arrive at:
\begin{equation}\label{eq: F_g_est}
\sup_X F_g \leq C d_1(0,g.0) + C.
\end{equation}
To obtain the Laplacian estimate, we start with Yau's calculation (for a survey, see \cite[Proposition 4.1.2]{BEG13}):
$$\textup{Tr}_{\omega_{g.0}} \big[i \ddbar  \log  \textup{Tr}_{\o} \o_{g.0} \big] \geq \frac{\textup{Tr}_\omega \big[ i\ddbar \log \big( \frac{\o^n_{g.0}}{\o^n}\big)\big]}{\textup{Tr}_\o \omega_{g.0}} - C \textup{Tr}_{\omega_{g.0}} \omega,$$
where $C>0$ only depends on $(X,\omega)$. Let $B:=2C+1$. Using the fact that $\textup{Ric }\omega_{g.0} \leq C \omega_{g.0}$, we can continue:
\begin{flalign*}
\textup{Tr}_{\omega_{g.0}} \big[i \ddbar  \big(\log  \textup{Tr}_{\o} \o_{g.0} - B g.0\big) \big] &\geq \frac{\textup{Tr}_\omega \big[- C \omega_{g.0} - C \omega\big]}{\textup{Tr}_\o \omega_{g.0}} - C \textup{Tr}_{\omega_{g.0}}\omega - B \textup{Tr}_{\omega_{g.0}} \big[i\ddbar g.0\big]\\
&\geq -\frac{nC}{\textup{Tr}_\o \omega_{g.0}} + (B-C)\textup{Tr}_{\omega_{g.0}} \omega -nB-C\\
& \geq  (B-2C)\textup{Tr}_{\omega_{g.0}} \omega -nB-C\\
& \geq \textup{Tr}_{\omega_{g.0}} \omega -C\geq \Big( \frac{\o^n_{g.0}}{\o^n} \Big)^{\frac{-1}{n-1}} \big( \textup{Tr}_\o \o_{g.0}\big)^{\frac{1}{n-1}} -C \\
&= F_g^{\frac{-1}{n-1}} \big( \textup{Tr}_\o \o_{g.0}\big)^{\frac{1}{n-1}}-C.
\end{flalign*}  
Let $x_0 \in X$ be the point where $\big(\log  \textup{Tr}_{\o} \o_{g.0} - B g.0\big)$ is maximized. Using the above estimate and \eqref{eq: F_g_est} we obtain that $\textup{Tr}_{\o}\o_{g.0}(x_0) \leq C d_1(0.g.0) + C$. Together with \eqref{eq: C_0_g_est} we arrive at $\sup_X \log(n + \Delta_\omega (g.0) ) \leq Cd_1(0,g.0) + C$.
\end{proof}

\section{Appendix}

Here we address two likely known facts about K\"ahler potentials with bounded Laplacian, whose proof we could not find in the literature. 

\begin{lemma}\label{lem: Laplace_aproxx} Let $u,u_j \in \mathcal H^{1, \bar 1}_\o$ and $B \in \Bbb R$. If $u_j \searrow u$ then 
\begin{equation}\label{eq: Laplace_semicont}
\liminf _j \textup{ess sup}_X (\log( n + \Delta_\omega u_j) - Bu_j)  \geq  \textup{ess sup}_X (\log( n + \Delta_\omega u) - Bu).
\end{equation}
\end{lemma}

\begin{proof} After picking subsequence, we can assume without loss of generality that the $\liminf$ on the left hand side   is actually a limit. Let $\delta \in \Bbb R$ such that $\log(n + \Delta_\omega u_j(x)) - B u_j(x) < \delta $ for a.e. $x \in X$ and $j \in \Bbb N$. To conclude, it is enough to show that 
\begin{equation}\label{eq: lem_ap_toprove}
\log(n + \Delta_\omega u) - B u \leq  \delta, \ \  \textup{a.e. on } X.
\end{equation}
By assumption, $\Delta_{\omega} u_j + n \leq e^{Bu_j+\delta}$ in the weak sense of postive measures on $X$. By Dini's lemma we have that $\|u_j - u \|_{C^0} \to 0$, hence passing to the weak limit we have that $\Delta_{\omega} u + n \leq e^{Bu+\delta}$, again in the weak sense of positive measures on $X$. Since all our measures have bounded densities, \eqref{eq: lem_ap_toprove} follows. 
\end{proof}

Complementing the above lemma, in the next result we point out that the quantity on the right hand side of \eqref{eq: Laplace_semicont} can be realized with an approriate decreasing sequence, constructed via the method of \cite{De94}. Let us recall some elements of this work. We denote by $\textup{exph}_x:T_x X \to X$ the ``quasiholomorpic exponential map" of $\omega$ (see \cite[Section 2]{De94}). Let $\chi: \Bbb R \to \Bbb R$ be an even non-negative smooth function supported in $[0,1]$ such that $\int_{\mathbb{C}^n} \chi(\|\xi\|^2)d\lambda(\xi)=1$. Given $u \in \textup{PSH}(X,\omega)$, one can introduce $u_\varepsilon \in C^\infty(X)$ by the following formula:
$$u_\varepsilon(x) := \frac{1}{\varepsilon^{2n}}\int_{T_x X} u(\textup{exph}_x(\xi)) \chi\bigg(\frac{|\xi|^2}{\varepsilon^2}\bigg) d\lambda(\xi),$$
where $d\lambda$ is the Lebesgue measure on $T_x X$ with respect to $\omega$.

\begin{proposition}\label{prop: Laplace_realiz_aproxx} Let $u \in \mathcal H^{1, \bar 1}_\o$ and $B \in \Bbb R$. There exists $u_j \in \mathcal H_\o$ such that $u_j$ converges to  $u$ decreasingly (and uniformly by Dini's lemma) and 
\begin{equation}
\lim_j \sup_X (\log( n + \Delta_\omega u_j) - Bu_j) = \textup{ess sup}_X (\log( n + \Delta_\omega u) - Bu).
\end{equation}
\end{proposition}

\begin{proof} By possibly rescaling $u$ with a small constant, we can assume that there exists $\delta >0$ such that $\omega_u \geq \delta \omega$. In particular, it follows from the estimate of \cite[Theorem 4.1]{De94} that for small enough $\varepsilon >0$ we actually have that $u_\varepsilon \in \mathcal H_\o$. Moreover, $\|u_\varepsilon - u \|_{C^{0}} \to 0$. Also, it follows from \cite[Theorem 3.8]{De94} that
$$i\ddbar {u_\varepsilon}(\zeta,\zeta) = \int_{T_x X} i\ddbar u\big|_{\textup{exph}_x(\varepsilon\xi)}(\zeta,\zeta)  \chi(|\xi|^2)d\lambda(\xi) + O(|\varepsilon|)(\zeta,\zeta), \ \ \ \zeta \in T_x X, \ x \in X.$$
Consequently, by an elementary local calculation, we have that:
$$\lim_{\varepsilon \to 0} \sup_X (\log( n + \Delta_\omega u_\varepsilon) - Bu_\varepsilon) = \textup{ess sup}_X (\log( n + \Delta_\omega u) - Bu).
$$
After possibly adding small constants to $u_\varepsilon$, we   can construct the decreasing sequence desired.
\end{proof}

\footnotesize
\let\OLDthebibliography\thebibliography 
\renewcommand\thebibliography[1]{
  \OLDthebibliography{#1}
  \setlength{\parskip}{1pt}
  \setlength{\itemsep}{1pt}
}

\bigskip
\normalsize
\noindent {\sc University of Maryland} \\
\noindent {\tt tdarvas@math.umd.edu} \vspace{0.1in}\\
\noindent {\sc Universit\'e Paris-Sud}\\
{\tt hoang-chinh.lu@u-psud.fr}
\end{document}